\documentclass[12pt]{amsart}
\usepackage{latexsym}
\usepackage{amssymb}
\usepackage{amscd}
\usepackage{mathrsfs}
\usepackage{rotating}

\renewcommand\a{\alpha}
\renewcommand\b{\beta}
\newcommand\g{\gamma}
\renewcommand\d{\delta}
\newcommand\la{\lambda}

\newcommand\e{\eta}
\renewcommand\th{\theta}

\newcommand\s{\sigma}

\renewcommand\t{\tau}

\newcommand\Om{\Omega}

\newcommand\vD{\varDelta}

\newcommand\vL{\varLambda}

\newcommand\ve{\varepsilon}

\newcommand\BQ{\mathbf Q}

\newcommand\BZ{\mathbf Z}
\newcommand\BM{\mathbf M}

\newcommand\Bm{\mathbf m}

\newcommand\Bt{\mathbf t}
\newcommand\Bu{\mathbf u}

\newcommand\Bi{\mathbf i}

\newcommand\Bbe{\boldsymbol\beta}

\newcommand\Bdel{\boldsymbol\delta}
\newcommand\Bla{\boldsymbol\lambda}

\newcommand\Bmu{\boldsymbol\mu}
\newcommand\Bnu{\boldsymbol\nu}

\newcommand\Bth{\boldsymbol{\th}}

\newcommand\SA{\mathscr{A}}
\newcommand\SB{\mathscr{B}}

\newcommand\SM{\mathscr{M}}
 
\newcommand\SO{\mathscr{O}}
\newcommand\SP{\mathscr{P}}

\newcommand\SX{\mathscr{X}}

\newcommand\iv{^{-1}}
\newcommand\wh{\widehat}
\newcommand\wt{\widetilde}

\newcommand\lra{\leftrightarrow}

\newcommand\lv{\prec}
\newcommand\gv{\succ}
\newcommand\lve{\preceq}
\newcommand\gve{\succeq}

\newcommand\id{\operatorname{id}}

\newcommand\lp{\operatorname{\!\langle\!}}
\newcommand\rp{\operatorname{\!\rangle\!}}

\newcommand{\isom}{\,\raise2pt\hbox{$\underrightarrow{\sim}$}\,}
\numberwithin{equation}{section}

\newtheorem{thm}{Theorem}[section]
\newtheorem{lem}[thm]{Lemma}
\newtheorem{cor}[thm]{Corollary}
\newtheorem{prop}[thm]{Proposition}

\def \para#1{\par\medskip\textbf{#1}
              \addtocounter{thm}{1}}

\def \remark#1{\par\medskip\noindent
                \textbf{Remark #1}
                \addtocounter{thm}{1}}

\begin{document}
\setlength{\baselineskip}{4.9mm}
\setlength{\abovedisplayskip}{4.5mm}
\setlength{\belowdisplayskip}{4.5mm}
%%%
%%%
\renewcommand{\theenumi}{\roman{enumi}}
\renewcommand{\labelenumi}{(\theenumi)}
\renewcommand{\thefootnote}{\fnsymbol{footnote}}
%%%
\renewcommand{\thefootnote}{\fnsymbol{footnote}}
%%%
\allowdisplaybreaks[2]
%\NoBlackBoxes
\parindent=20pt
%\addtocounter{section}{1}

%%%%%%%%%%%%%%%%%%%%
%%%%%%%%%%%%%%%%%%%%%%%%%%%%%%%%%%%
\pagestyle{myheadings}
\medskip
\begin{center}
 {\bf Kostka functions associated to complex reflection groups \\
and a conjecture of Finkelberg-Ionov} 
\end{center}

\par\bigskip

\begin{center}
Toshiaki Shoji
\\  
\vspace{0.5cm}

\end{center}

\title{}

\begin{abstract}
Kostka functions $K^{\pm}_{\Bla,\Bmu}(t)$, indexed by $r$-partitions $\Bla, \Bmu$ of $n$, 
are a generalization of Kostka polynomials $K_{\la,\mu}(t)$ indexed by partitions 
$\la, \mu$ of $n$.  It is known that Kostka polynomials have an interpretation in terms
of Lusztig's partition function. Finkelberg and Ionov defined alternate functions 
$K_{\Bla, \Bmu}(t)$ by using an analogue of Lusztig's partition function, and showed that 
$K_{\Bla,\Bmu}(t) \in \BZ_{\ge 0}[t]$ for generic $\Bmu$ by making use of a coherent realization.
They conjectured that $K_{\Bla,\Bmu}(t)$ coincide with $K^-_{\Bla,\Bmu}(t)$.  
In this paper, we show that their conjecture holds. We also discuss the multi-variable
version, namely, $r$-variable Kostka functions $K^{\pm}_{\Bla,\Bmu}(t_1, \dots, t_r)$.
\end{abstract}

\maketitle
%\markboth{SHOJI}{KOSTKA FUNCTIONS}
\pagestyle{myheadings}

\section{Introduction }
Let $K_{\la,\mu}(t) \in \BZ[t]$ be Kostka polynomials indexed by partitions 
$\la,\mu$ of $n$. It is known by [M, III, Example 4] that Kostka polynomials 
have an interpretation in terms of Lusztig's partition function.  
(Actually, Lusztig defined a partition function in [L1], and conjectured 
 a formula on a $q$-analogue 
of the weight multiplicities for semisimple Lie algebras in terms of 
his partition function, as a generalization of the above result which corresponds to
the case of type $A$.  Soon after that the conjecture was proved by [K].) 
Let $\SP_{n,r}$ be the set of $r$-tuples of partitions 
$\Bla = (\la^{(1)}, \dots, \la^{(r)})$ 
such that $\sum_{i = 1}^r |\la^{(i)}| = n$.  
In [S1, S2], as a generalization of the classical Kostka polynomials, Kostka functions 
$K^{\pm}_{\Bla,\Bmu}(t)$ attached to $\Bla, \Bmu \in \SP_{n,r}$ were introduced.
(In general, there exist two types, ``$+$'' and ``$-$'' types.  If $r = 1$ or 2, 
$K^+_{\Bla,\Bmu}(t) = K^-_{\Bla,\Bmu}(t)$. If $r = 1$, they coincide with the classical 
Kostka polynomials.)   
A priori, they are rational functions in $\BQ(t)$, 
and the construction depends on the choice of a total order on $\SP_{n,r}$.     
$K^{\pm}_{\Bla,\Bmu}(t)$ are called Kostka functions associated to complex 
reflection groups, or $r$-Kostka functions, in short 
(see [S1] for the relationship with the complex reflection 
group $S_n\ltimes (\BZ/r\BZ)^n$).   
\par
In [FI], Finkelberg and Ionov introduced polynomials $K_{\Bla, \Bmu}(t) \in \BZ[t]$ 
attached to $\Bla, \Bmu \in \SP_{n,r}$, by using an analogue of Lusztig's partition 
function on $GL_{mr}$, where we choose an integer $m$ such that 
the number of parts of $\la^{(i)}, \mu^{(i)}$ is smaller than $m$ for each $i$. 
They proved, in the case where $\Bmu$ is regular (see 7.6 for the precise 
definition), that $K_{\Bla,\Bmu}(t) \in \BZ_{\ge 0}[t]$, 
by showing the higher cohomology vanishing 
$H^{i > 0}(\SX, \wt\SO(\Bmu)) = 0$, where $\SX$ is a certain $(GL_m)^r$-equivariant 
vector bundle over the flag variety $\SB$ of $(GL_m)^r$, and 
$\wt\SO(\Bmu)$ is the pull-back of the $(GL_m)^r$-equivariant ample line bundle 
$\SO(\Bmu)$ over $\SB$ associated to $\Bmu \in \SP_{n,r}$.
(In fact, they proved the higher cohomology vanishing by showing the Frobenius splitting 
of $\SX$.) In turn, the higher cohomology vanishing for general $\Bmu$ was recently proved 
by Hu [H]. Hence the positivity property for $K_{\Bla,\Bmu}(t)$ now holds 
without any restriction. 
Their result is a natural 
generalization of the coherent realization of the classical Kostka polynomials 
due to Brylinski [B].
On the other hand, the vector bundle $\SX$ is nothing but (a special case of ) 
Lusztig's iterated covolution diagram ([L2]) associated to the cyclic quiver of $r$-vertices. 
In this direction, Orr and Shimozono [OS] constructed a wider class of polynomials, 
as a generalization of $K_{\Bla,\Bmu}(\Bt)$ of [FI], by making use of Lusztig's 
iterated covolution diagram associated to arbitrary quivers. 
\par
Finkelberg and Ionov conjectured in [FI] that $K_{\Bla,\Bmu}(t)$ coincide with our 
$K^-_{\Bla,\Bmu}(t)$. More generally,   
they construct in [FI] polynomials $K_{\Bla, \Bmu}(t_1, \dots, t_r) \in \BZ[t_1, \dots, t_r]$,
a multi-variable version of $K_{\Bla,\Bmu}(t)$,
by making use of Lusztig's partition function. 
Those polynomials have a property that $K_{\Bla,\Bmu}(t_1, \dots, t_r)$
coincides with $K_{\Bla,\Bmu}(t)$ if $t_1 = \cdots = t_r = t$. 
Inspired by their work, we generalize our $r$-Kostka functions to the 
multi-variable case.  In the one-variable case, $K^{\pm}_{\Bla,\Bmu}(t)$ are defined 
as the coefficients of the expansion of Schur functions $s_{\Bla}(x)$ 
in terms of the Hall-Littlewood
functions $P^{\pm}_{\Bmu}(x;t)$, where $x = (x^{(1)}, \dots, x^{(r)})$ are $r$ types 
of infinitely many variables $x^{(k)} = x^{(k)}_1, x^{(k)}_2, \dots$.         
Hence we generalize the definition of Hall-Littlewood functions to the multi-parameter
case $P^{\pm}_{\Bla}(x;\Bt)$ with $\Bt = (t_1, \dots, t_r)$, and define 
$K^{\pm}_{\Bla, \Bmu}(\Bt)$ as the coefficients of the expansion of $s_{\Bla}(x)$ 
in terms of $P^{\pm}_{\Bla}(x;\Bt)$. 
Note that the construction of Hall-Littlewood functions depends on the choice of the 
total order which is compatible with the dominance order on $\SP_{n,r}$, and they are
symmetric functions with respect to $x^{(1)}, \dots, x^{(r)}$ with coefficients in 
$\BQ(\Bt)$. 
We show that both of $K^{\pm}_{\Bla,\Bmu}(\Bt)$ have an interpretation in terms of 
an analogue of Lusztig's partition function, and that $K^-_{\Bla,\Bmu}(\Bt)$ 
coincides with their $K_{\Bla,\Bmu}(\Bt)$, which proves their conjecture in a generalized 
form. 
\par
The main step for the proof is to establish a closed formula for 
Hall-Littlewood functions as given in [M, III, 2] for the
classical case.  By using this formula, one can show that  Hall-Littlewood 
functions are actually independent of the choice of the total order, and are
symmetric functions with coefficients in $\BZ[\Bt]$.  This implies that 
$K^{\pm}_{\Bla,\Bmu}(\Bt) \in \BZ[\Bt]$, and  
they are independent of the choice of the total order. 
\par
Note that in establishing the closed formula for Hall-Littlewood functions, 
the multi-variable setting is essential. 
Even if one is interested only in the one-variable case, our discussion does not
work without multi-variable setting. 
\par
In December of 2015, Michael Finkelberg gave a lecture concerning their conjecture 
at the conference in Shanghai, Chongming Island.  This work arose from his 
interesting talk there, and from his question on the stability of Kostka functions 
on the occasion of the conference at Besse-et-St-Anastaise in 2013. 
The author is very grateful for him for stimulating discussions.  

\par\bigskip\noindent
{\bf Contents}
\par\medskip\noindent
1. Introduction  \\
2. Hall-Littlewood functions with multi-parameter  \\
3. Comparison of Hall-Littlewood functions for different $r$ \\
4. Closed formula for Hall-Littlewood functions \\
5. Closed formula for Hall-Littlewood functions -- continued  \\
6. Closed formula for Hall-Littlewood functions -- ``$+$''case  \\
7. A conjecture of Finkelberg-Ionov       

\section{Hall-Littlewood functions with multi-parameter}

\para{2.1.}
First we recall basic properties of Hall-Littlewood functions and Kostka polynomials
in the original setting, following [M]. 
Let $\SP_n$ be the set of partitions $\la = (\la_1, \dots, \la_k)$ with $\la_i \ge 0$ 
such that 
$|\la| = \sum_i \la_i = n$.
For a partition $\la$, the length $l(\la)$ of $\la$ is defined as the number of $\la_i$ 
such that $\la_i \ne 0$.  
\par
Let $\vL = \vL(y) = \bigoplus_{n \ge 0}\vL^n$ be the ring of symmetric functions 
over $\BZ$ with respect to the variables $y = (y_1, y_2, \dots)$, where $\vL^n$ denotes the 
free $\BZ$-module of symmetric functions of degree $n$. 
For each $\la \in \SP_n$, the Schur function $s_{\la}(y) \in \vL$ is defined as follows; 
first choose finitely many variables 
$y_1, \dots, y_m$ such that $m \ge l(\la)$, and define the Schur polynomial 
$s_{\la}(y_1, \dots, y_m) \in \BZ[y_1, \dots, y_m]$ by 
\begin{equation*}
s_{\la}(y_1, \dots, y_m) = \det(y_i^{\la_j + m - j})/\det(y_i^{m-j}).
\end{equation*}
$s_{\la}(y_1, \dots, y_m)$ satisfies the stability property 
\begin{equation*}
s_{\la}(y_1, \dots, y_m, y_{m+1})|_{y_{m+1} = 0} = s_{\la}(y_1, \dots, y_m),
\end{equation*}
and one can define $s_{\la}(y)$ by 
\begin{equation*}
s_{\la}(y) = \lim_{\substack{\longleftarrow \\ m }}s_{\la}(y_1, \dots, y_m).
\end{equation*}
Then 
$\{ s_{\la} \mid \la \in \SP_n \}$ gives a $\BZ$-basis of $\vL^n$. 

\para{2.2}
We fix $m$, and consider a partition $\la = (\la_1, \dots, \la_m) \in \SP_n$ 
such that $l(\la) \le m$. We denote $\la$ as 
$\la = (0^{m_0}, 1^{m_1}, 2^{m_2}, \dots)$, where $m_i = \sharp\{ j \mid \la_j = i\}$
for $i = 0,1, 2, \dots$.   
Let $t$ be an indeterminate. We define 
a polynomial $v_{\la}(t) \in \BZ[t]$ as follows; 
for each integer $k \ge 1$, we define $v_k(t)$ by 
\begin{equation*}
\tag{2.2.1}
v_k(t) = \sum_{w \in S_k}t^{\ell(w)} = \prod_{i=1}^k\frac{1 - t^i}{1 - t},
\end{equation*}
where $\ell(w)$ is the length function of the symmetric group $S_k$ of degree $k$, 
and put $v_k(t) = 1$ for $k = 0$.
Set
\begin{equation*}
\tag{2.2.2}
v_{\la}(t) = \prod_{i \ge 0} v_{m_i}(t). 
\end{equation*}
\par
The symmetric group $S_m$ acts on the set of variables 
$\{y_1, \dots, y_m\}$ as permutations.  
For $l(\la) \le m$, we define the Hall-Littlewood 
polynomial $P_{\la}(y_1, \dots, y_m;t) \in \BZ[y_1, \dots, y_m;t]$ by 

\begin{equation*}
\tag{2.2.3}
P_{\la}(y_1, \dots, y_m;t) = \frac{1}{v_{\la}(t)}
               \sum_{w \in S_m}w\biggl( y^{\la}\prod_{i < j}\frac{y_i - ty_j}{y_i - y_j}\biggr), 
\end{equation*}
where we use the standard notation $y^{\a} = y_1^{\a_1}y_2^{\a_2}\cdots y_m^{\a_m}$ 
for $\a = (\a_1, \dots, \a_m) \in \BZ^m$. 
\par
Let $\vL[t] = \BZ[t]\otimes_{\BZ}\vL$ be the ring of symmetric functions with 
coefficients in $\BZ[t]$.  We have $\vL[t] = \bigoplus_{n \ge 0}\vL^n[t]$, 
where $\vL^n[t] = \BZ[t]\otimes_{\BZ}\vL^n$.
The Hall-Littlewood polynomial has the stability property, and one can define 
the Hall-Littlewood function $P_{\la}(y;t) \in \vL[t]$ by taking $m \mapsto \infty$. 
$\{ P_{\la}(y;t) \mid \la \in \SP_{n,r} \}$ gives a $\BZ[t]$-basis of the free 
$\BZ[t]$-module $\vL^n[t]$. 
Another type of Hall-Littlewood function $Q_{\la}(y;t)$ is defined by 
$Q_{\la}(y;t) = b_{\la}(t)P_{\la}(y;t)$, where 
$b(t) = v_{\la}(t)(1-t)^{l}/v_{m_0}(t) \in \BZ[t]$
with $l = l(\la)$. 
$\{ Q_{\la} \mid \la \in \SP_n\}$ gives a $\BQ(t)$-basis of 
$\vL^n_{\BQ}(t) = \BQ(t)\otimes_{\BZ}\vL^n$.

\para{2.3.}
For $\la, \mu \in \SP_n$, the Kostka polynomials $K_{\la,\mu}(t) \in \BZ[t]$ are 
defined by the formula
\begin{equation*}
\tag{2.3.1}
s_{\la}(y) = \sum_{\mu \in \SP_n}K_{\la,\mu}(t)P_{\mu}(y;t).
\end{equation*}

\par
We define a partial order $\xi \le \eta$, the so-called dominance order, on $\BZ^m$   
by the condition, for 
$\xi = (\xi_1, \dots, \xi_m), \eta = (\eta_1, \dots, \eta_m) \in \BZ^m$,
\begin{equation*}
\tag{2.3.2}
\sum_{i = 1}^k \xi_i \le \sum_{i=1}^k\eta_i \quad \text{ for } k = 1, \dots, m.
\end{equation*}  
\par
For each partition $\la = (\la_1, \dots, \la_m)$, we define an integer $n(\la)$ 
by 
\begin{equation*}
\tag{2.3.3}
n(\la) = \sum_{i=1}^m(i-1)\la_i.  
\end{equation*}
It is known that $K_{\la,\mu}(t) = 0$ unless $\mu \le \la$, and in which case,
$K_{\la,\mu}(t)$ is monic of degree $n(\mu) - n(\la)$. 
\par
For any integer $s \ge 1$, we define a function $q_s(y_1, \dots, y_m ;t)$ by 
\begin{equation*}
\tag{2.3.4}
q_s(y_1, \dots, y_m;t) = (1 - t)\sum_{i = 1}^my_i^s\prod_{j \ne i}\frac{y_i - ty_j}{y_i - y_j},
\end{equation*}
and put $q_s = 1$ for $s = 0$. 
The generating function for $q_s$ is given as follows ([M, III, (2.10)]). 
Let $u$ be another indeterminate.  Then we have 
\begin{equation*}
\tag{2.3.5}
\sum_{s = 0}^{\infty}q_s(y_1, \dots, y_m;t)u^s = 
             \prod_{i= 1}^m\frac{1 - tuy_i}{1  - uy_i}.
\end{equation*}
For a partition $\la = (\la_1, \dots, \la_m) \in \SP_n$, we define a function $q_{\la}$
by 
\begin{equation*}
q_{\la}(y_1, \dots, y_m;t) = \prod_{i = 1}^m q_{\la_i}(y_1, \dots, y_m ;t).
\end{equation*}
By taking $m \mapsto \infty$, $q_{\la}(y_1, \dots, y_m ;t)$ defines $q_{\la}(y;t) \in \vL^n[t]$.
Then $Q_{\la}$ has an expansion by $q_{\mu}$,
$Q_{\la} = q_{\la} + \sum_{\mu >  \la}a_{\la,\mu}(t)q_{\mu}$, 
with $a_{\la,\mu}(t) \in \BZ[t]$. 
Hence $\{ q_{\la} \mid \la \in \SP_n \}$ gives a $\BQ(t)$-basis of 
$\vL^n_{\BQ}(t)$. 

\para{2.4.}
We fix an integer $r \ge 2$, and consider the $r$ types of variables 
$x = (x^{(1)}, \dots, x^{(r)})$, where $x^{(k)}$ stands for the infinitely many 
variables $x^{(k)}_1, x^{(k)}_2, \dots$.  
We consider the ring of symmetric functions 
$\Xi = \Xi(x) = \vL(x^{(1)})\otimes\cdots\otimes \vL(x^{(r)})$, 
symmetric with respect to each variable $x^{(k)}$. 
We have $\Xi = \bigoplus_{n \ge 0}\Xi^n$, where $\Xi^n$ is the free $\BZ$-module 
consisting of homogeneous symmetric functions of degree $n$. 
\par
Let $\SP_{n,r}$ be the set of $r$-tuple of partitions 
$\Bla = (\la^{(1)}, \dots, \la^{(r)})$ such that $\sum_{k=1}^r|\la^{(k)}| = n$. 
For $\Bla \in \SP_{n,r}$, we choose an integer $m$ such that $m \ge l(\la^{(k)})$ 
for any $k$, and consider the finitely many variables 
$\{ x^{(k)}_i \mid 1\le k \le r, 1\le i \le m \}$.
We prepare the index set 
\begin{equation*}
\tag{2.4.1}
\SM = \SM(m) = \{ (k,i) \mid 1 \le k \le r, 1 \le i \le m \},
\end{equation*}
and write $x^{(k)}_i$ as $x_{\nu}$ for $\nu = (k,i) \in \SM$. 
We denote by $x_{\SM}$ the set of variables $\{ x_{\nu} \mid \nu \in \SM\}$.   
\par
 We define a polynomial $s_{\Bla}(x_{\SM})$ by 

\begin{equation*}
s_{\Bla}(x_{\SM}) = \prod_{k=1}^rs_{\la^{(k)}}(x^{(k)}_1, \dots, x^{(k)}_m) \in \BZ[x_{\SM}].
\end{equation*}

$s_{\Bla}(x_{\SM})$ has the stability property with respect to the operation 
$x^{(1)}_{m+1} = \cdots = x^{(r)}_{m+1} = 0$ in $\SM(m+1)$, 
and by taking $m \mapsto \infty$, one can 
define $s_{\Bla}(x) \in \Xi$.  Then  
$\{ s_{\Bla}(x) \mid \Bla \in \SP_{n,r}\}$ gives a basis of $\Xi^n$.  
\par
For a partition $\la \in \SP_n$, the monomial symmetric polynomial  
$m_{\la}(y_1, \dots, y_m) \in \BZ[y_1, \dots, y_m]$ is defined for $m \ge l(\la)$.  
For $\Bla \in \SP_{n,r}$, we define 
$m_{\Bla}(x_{\SM})$ by 
\begin{equation*}
m_{\Bla}(x_{\SM}) = \prod_{k=1}^r m_{\la^{(k)}}(x_1^{(k)}, \dots, x^{(k)}_m) \in \BZ[x_{\SM}].
\end{equation*} 
By taking $m \mapsto \infty$, one can define 
$m_{\Bla}(x) \in \Xi^n$, and $\{ m_{\Bla}(x) \mid \Bla \in \SP_{n,r}\}$ 
gives a basis of $\Xi^n$. 

\para{2.5.}
For any integer $s \ge 1$ we define a function 
$q^{(k)}_{s, \pm}(x;t)$ (for $x = x_{\SM}$) by 
\begin{equation*}
\tag{2.5.1}
q^{(k)}_{s,\pm}(x;t) = \sum_{i = 1}^m(x_i^{(k)})^{s-1}
          \frac{\prod_{j \ge 1}(x^{(k)}_i - tx_j^{(k\mp 1)})}
         {\prod_{j \ne i}(x_i^{(k)} - x_j^{(k)})},
\end{equation*} 
where we regard $k \in \BZ/r\BZ$, 
and put $q^{(k)}_{s, \pm}(x;t) = 1$ for $s = 0$.
\par
Let $u$ be another indeterminate. As in the proof of [S1, Lemma 2.3], 
by using the Lagrange's interpolation formula
\begin{equation*}
\prod_{i=1}^m\frac{1 - tux_i^{(k \mp 1)}}{1 - ux_i^{(k)}}
    = 1 + \sum_{i=1}^m\frac{ux_i^{(k)} - tux_i^{(k \mp 1)}}{1 - ux_i^{(k)}}
               \prod_{j \ne i}\frac{x_i^{(k)} - tx_j^{(k \mp 1)}}{x_i^{(k)} - x_j^{(k)}},
\end{equation*}
one can prove the formula

\begin{equation*}
\tag{2.5.2}
\sum_{s = 0}^{\infty}q^{(k)}_{s, \pm}(x;t)u^s = \prod_{i=1}^m 
         \frac{1 - tux_i^{(k \mp 1)}}{1 - ux_i^{(k)}}.
\end{equation*}
It follows from (2.5.2) that $q_{s,\pm}^{(k)}(x;t) \in \BZ[x_{\SM};t]$, and 
symmetric with respect to the variables $x^{(k)}, x^{(k \mp 1)}$.  Moreover, it satisfies the 
stability property.   
\par
Let $\Xi[t] = \BZ[t]\otimes_{\BZ}\Xi$ be the ring of symmetric functions in $\Xi$ 
with coefficients in $\BZ[t]$.  Put $\Xi^n[t] = \BZ[t]\otimes_{\BZ}\Xi^n$. 
More generally, we consider the multi-parameter case. 
Let $\Bt = (t_1, \dots, t_r)$ be $r$-parameters, and consider $\BZ[\Bt] = \BZ[t_1, \dots, t_r]$. 
Put $\Xi[\Bt] = \BZ[\Bt] \otimes_{\BZ}\Xi$. 
We have $\Xi[\Bt] = \bigoplus_{n \ge 0}\Xi^n[\Bt]$, where 
$\Xi^n[\Bt] = \BZ[\Bt]\otimes_{\BZ} \Xi^n$. 
\par
For $\Bla \in \SP_{n,r}$, we define polynomials
$q^{\pm}_{\Bla}(x_{\SM};\Bt) \in \BZ[x_{\SM};\Bt] = \BZ[x_{\SM}; t_1, \dots, t_r]$ by 

\begin{align*}
\tag{2.5.3}
q^{\pm}_{\Bla}(x_{\SM};\Bt) &= 
   \prod_{k \in \BZ/r\BZ}\prod_{i = 1}^m q^{(k)}_{\la^{(k)}_i, \pm}(x;t_{k-c}), \\
\end{align*}
where $c = 1$ for the ``$+$''-case, and $c = 0$ for the ``$-$''-case. 
Then $q^{\pm}_{\Bla}(x_{\SM};\Bt)$ satisfies the stability condition, and one can define 
$q^{\pm}_{\Bla}(x;\Bt) \in \Xi^n[\Bt]$.  
Note that if $t_1 = \cdots = t_r = t$, $q^{\pm}_{\Bla}(x;\Bt)$ coincides with 
$q^{\pm}_{\Bla}(x;t)$ defined in [S1, 2.4].
\par
Put $\Xi^n_{\BQ}(t) = \BQ(t) \otimes_{\BZ} \Xi^n$, and 
$\Xi^n_{\BQ}(\Bt) = \BQ(\Bt) \otimes_{\BZ} \Xi^n$. 
It is known by [S1, (4.7.2)] that $\{ q^{\pm}_{\Bla}(x;t) \mid \Bla \in \SP_{n,r} \}$ 
gives a $\BQ(t)$-basis of $\Xi^n_{\BQ}(t)$. 
The analogous fact holds also in the multi-parameter case.

\begin{lem}   %%%%   Lemma 2.7
$\{ q^{\pm}_{\Bla}(x;\Bt) \mid \Bla \in \SP_{n,r} \}$ gives a $\BQ(\Bt)$-basis 
of $\Xi^n_{\BQ}(\Bt)$. 
\end{lem}

\begin{proof}
Since $\{ s_{\Bmu}(x) \mid \Bmu \in \SP_{n,r}\}$ is a $\BZ[\Bt]$-basis of 
$\Xi^n[\Bt]$ and $q_{\Bla}^{\pm}(\Bt) \in \Xi^n[\Bt]$, 
$q^{\pm}_{\Bla}(x;\Bt)$ can be written as a linear combination of
$s_{\Bmu}(x)$.  Let $A(\Bt) = (a_{\Bla,\Bmu}(\Bt))$ be the corresponding matrix
with $a_{\Bla,\Bmu}(\Bt) \in \BZ[\Bt]$.
Let $A(t)$ be the matrix obtained from $A(\Bt)$ by putting $t_1 = \dots = t_r = t$.  
Then $A(t)$ is a non-singular matrix by the above remark.  
Hence $A(\Bt)$ is also non-singular, and 
$q^{\pm}_{\Bla}(x;\Bt)$ gives a basis of $\Xi^n_{\BQ}(\Bt)$. 
\end{proof}

\para{2.7.}
We consider two types 
of (infinitely many) variables $x = (x^{(1)}, \dots, x^{(r)})$ and 
$y = (y^{(1)}, \dots, y^{(r)})$, and put  
\begin{equation*}
\Om(x,y;\Bt) = \prod_{k \in \BZ/r\BZ}\prod_{i,j}
               \frac{1 - t_kx^{(k)}_iy^{(k+1)}_j}{1 - x^{(k)}_iy^{(k)}_j}.
\end{equation*}

The following formula is a multi-parameter version of [S1, (2.5.1)].  
The proof is done by an entirely similar way, and we omit it. 

\begin{prop}   %%%%   Prop. 2.8
Under the notation above, we have

\begin{align*}
\tag{2.8.1}
\Om(x,y;\Bt) &= \sum_{\Bla}q_{\Bla}^+(x;\Bt)m_{\Bla}(y)
            = \sum_{\Bla}m_{\Bla}(x)q^-_{\Bla}(y;\Bt). 
\end{align*}
\end{prop}  

\remark{2.9.}
In the one-parameter case, another expression of $\Om(x,y;t)$ involving power sum symmetric 
functions $p_{\Bla}(x)$ was proved in [S1, (2.5.2)].  However, we don't have a generalization 
of (2.5.2) in [S1] in the multi-parameter case.

\para{2.10.}
We define a non-degenerate bilinear form  
$\lp \ , \ \rp : \Xi^n_{\BQ}(\Bt) \times \Xi^n_{\BQ}(\Bt) \to \BQ(\Bt)$ by 

\begin{equation*}
\tag{2.10.1}
\lp q^+_{\Bla}(x;\Bt), m_{\Bmu}(x) \rp = \d_{\Bla,\Bmu}
\end{equation*}
for $\Bla, \Bmu \in \SP_{n,r}$. 
By using a similar argument as in [M, I,4], (2.8.1) implies that

\begin{equation*}
\tag{2.10.2}
\lp m_{\Bla}(x), q^-_{\Bmu}(x;\Bt) \rp = \d_{\Bla,\Bmu}.
\end{equation*}

\par
Let $\SA$ be the $\BQ$-subalgebra of $\BQ(\Bt)$ consisting of rational functions 
$f/g$ such that $g(\bold 0) \ne 0$, where $\bold 0 = (0, \dots, 0)$. 
Put $\SA \otimes_{\BZ}\Xi^n = \Xi^n_{\SA}$.
Then $\Xi^n_{\SA}|_{\Bt = \bold 0} = \Xi^n_{\BQ}$. 
By a similar argument as in [M, I,4], one can show 

\begin{equation*}
\tag{2.10.3}
\Om(x,y; \bold 0) = \prod_{k \in \BZ/r\BZ}\prod_{i,j}(1 - x^{(k)}_iy^{(k)}_j)\iv
                  = \sum_{\Bla}s_{\Bla}(x)s_{\Bla}(y). 
\end{equation*}
Hence if we define the symmetric bilinear form $\lp \ , \ \rp_{\,0}$ on $\Xi^n_{\BQ}$ by 
$\lp s_{\Bla}, s_{\Bmu} \rp_{\, 0} = \d_{\Bla, \Bmu}$ for $\Bla, \Bmu \in \SP_{n,r}$, 
the restriction of $\lp \ , \ \rp$ on $\Xi^n_{\SA}$ gives rise to the form 
$\lp \ , \ \rp_{\, 0}$ on $\Xi^n_{\BQ}$ by taking $\Bt \mapsto \bold 0$. 
\para{2.11.}
For $\Bla = (\la^{(1)}, \dots, \la^{(r)}) \in \SP_{n,r}$ with 
$\la^{(k)} = (\la^{(k)}_1, \dots, \la^{(k)}_m)$, we define 
$c(\Bla) \in \BZ_{\ge 0}^{rm}$ by 

\begin{equation*}
c(\Bla) = (\la^{(1)}_1, \dots, \la^{(r)}_1, \la^{(1)}_2, \dots, \la^{(r)}_2, 
           \dots, \la^{(1)}_m, \dots, \la^{(r)}_m ).
\end{equation*}
\par\medskip
\noindent
We define a partial order on $\SP_{n,r}$ by the condition, for $\Bla, \Bmu \in \SP_{n,r}$,  
$\Bla \le \Bmu$ if $c(\Bla) \le c(\Bmu)$ with respect to the dominance order on $\BZ^{rm}$. 
The partial order $\Bla \le \Bmu$ is called the dominance order on $\SP_{n,r}$.  
\par
In the remainder of this section, we fix a total order $\Bla \lve \Bmu$ on $\SP_{n,r}$
which is compatible with the dominance order $\Bla \le \Bmu$. 
\par
By making use of the bilinear form $\lp \ , \ \rp$, we shall construct
Hall-Littlewood functions with multi-parameter $P^{\pm}_{\Bla}(x;\Bt)$.
The following result is an analogue of [S1, Proposition 4.8].

\begin{prop}   %%%%  Prop 2.12 
For each $\Bla \in \SP_{n,r}$, there exists a unique function 
$P^{\pm}_{\Bla}(x;\Bt) \in \Xi^n_{\SA}$
satisfying the following properties. 
\begin{enumerate}
\item 
$P^{\pm}_{\Bla}(x;\Bt)$ can be expressed as 

\begin{equation*}
P^{\pm}_{\Bla}(x;\Bt) = s_{\Bla}(x) + \sum_{\Bmu \lv \Bla}u^{\pm}_{\Bla,\Bmu}(\Bt)s_{\Bmu}(x),
\end{equation*}
where $u^{\pm}_{\Bla,\Bmu}(\Bt) \in \SA$. 
\item
$\lp P^+_{\Bla}, P^-_{\Bmu} \rp = 0$ if $\Bla \ne \Bmu$. 
\item 
$P^{\pm}_{\Bla}(x; \bold 0) = s_{\Bla}(x)$. 
\end{enumerate}
\end{prop} 

\begin{proof}
We prove the proposition following the discussion in [S1, Remark 4.9].
We construct $P^{\pm}_{\Bla}(x;\Bt)$ satisfying the properties (i), (ii), (iii) 
by induction on the total order $\lve$ on $\SP_{n,r}$. 
Let $\Bla_0 = (-; \dots;-; (1^n))$.  $\Bla_0$ is the minimum element in $\SP_{n,r}$ 
with respect to 
$\le$, and so the minimum element with respect to $\lve$.  
By (i), $P^{\pm}_{\Bla_0}(x;\Bt)$ must coincide with $s_{\Bla_0}(x)$, which clearly
satisfies (iii). 
Take $\Bla \in \SP_{n,r}$, and assume,  
for any $\Bla',\Bla'' \lv \Bla$, that $P^{+}_{\Bla'}, P^-_{\Bla''}$ 
satisfying (i), (iii)  and 
$(\text{ii})'$ : $\lp P^+_{\Bla'}, P^-_{\Bla''} \rp = 0$ 
for $\Bla''\ne \Bla'$, was constructed.
Note that the condition (i) for $P^{\pm}_{\Bla}$ is equivalent to the condition 
\begin{equation*}
\tag{2.12.1}
P^{\pm}_{\Bla} = s_{\Bla} + \sum_{\Bla' \lv \Bla}d^{\pm}_{\Bla,\Bla'}P^{\pm}_{\Bla'}
\end{equation*}
with $d^{\pm}_{\Bla,\Bla'} \in \SA$.   
By taking the inner product with $P^-_{\Bmu}$ ($\Bmu \lv \Bla$) in (2.12.1), 
we  have a relation 

\begin{equation*}
\lp P^+_{\Bla}, P^-_{\Bmu} \rp = 
          \lp s_{\Bla}, P^-_{\Bmu} \rp \ + \  
             d^+_{\Bla,\Bmu}\lp  P^+_{\Bmu}, P^-_{\Bmu} \rp.
\end{equation*}
\par\medskip
\noindent
By (iii) and by 2.10, 
$\lp P^+_{\Bmu}, P^-_{\Bmu}\rp|_{\Bt = \bold 0} = \lp s_{\Bmu}, s_{\Bmu}\rp_{\, 0} = 1$. 
In particular, $\lp P^+_{\Bmu}, P^-_{\Bmu}\rp \ne 0$.  Hence
if we define $P^+_{\Bla}$ as in (2.12.1) with 
$d^+_{\Bla,\Bla'} = 
       - \lp s_{\Bla}, P^-_{\Bla'}\rp \, \lp P^+_{\Bla'}, P^-_{\Bla'} \rp\iv \in \SA$,  
we have $\lp P^+_{\Bla}, P^-_{\Bmu} \rp = 0$ for any $\Bmu \lv \Bla$. 
Since $\lp s_{\Bla}, P^-_{\Bla'}\rp |_{\Bt = \bold 0} = 
         \lp s_{\Bla}, s_{\Bla'} \rp_{\, 0} = 0$, 
we have $d^+_{\Bla, \Bla'}(\bold 0) = 0$. 
It follows that $P^+_{\Bla}(x;\bold 0) = s_{\Bla}(x)$. 
A similar argument shows, if we define $P^-_{\Bla}$ as in (2.12.1) with 
$d^-_{\Bla, \Bla'} = 
      - \lp P^+_{\Bla'}, s_{\Bla} \rp \, \lp P^+_{\Bla'}, P^-_{\Bla'} \rp\iv \in \SA$,  
that $P^-_{\Bla}$ satisfies the required condition.   Thus one can construct 
$P^{\pm}_{\Bla}$ satisfying (i), (ii), (iii).   
The uniqueness is clear from the construction.
\end{proof}

\para{2.13.}
The discussion in the proof of Proposition 2.12 shows that 
$b_{\Bla}(\Bt)  = \lp P^+_{\Bla}, P^-_{\Bla}\rp \iv \in \BQ(\Bt) - \{ 0 \}$. 
We define $Q^{\pm}_{\Bla}(x;\Bt) \in \Xi^n_{\SA}$ by 

\begin{equation*}
\tag{2.13.1}
Q^{\pm}_{\Bla}(x;\Bt) = b_{\Bla}(\Bt)P^{\pm}_{\Bla}(x;\Bt).
\end{equation*} 
Then we have 
\begin{equation*}
\tag{2.13.2}
\lp P^+_{\Bla}, Q^-_{\Bmu} \rp =  \lp Q^+_{\Bla}, P^-_{\Bmu}\rp = \d_{\Bla, \Bmu}. 
\end{equation*}
The sets $\{ P^{\pm}_{\Bla} \mid \Bla \in \SP_{n,r} \}$, 
$\{ Q^{\pm}_{\Bla} \mid \Bla \in \SP_{n,r}\}$ give $\BQ(\Bt)$-bases of $\Xi^n_{\BQ}(\Bt)$. 
$P^{\pm}_{\Bla}, Q^{\pm}_{\Bla}$ are called Hall-Littlewood functions with 
multi-parameter. 
Note that the formula (2.13.2) can be interpreted by using $\Om(x,y;\Bt)$ as follows;

\begin{align*}
\tag{2.13.3}
\Om(x,y;\Bt) &= \sum_{\Bla}b_{\Bla}(\Bt)P^+_{\Bla}(x;\Bt)P^-_{\Bla}(y;\Bt), \\
\tag{2.13.4}
\Om(x,y;\Bt) &=  \sum_{\Bla}P^+_{\Bla}(x;\Bt)Q^-_{\Bla}(y;\Bt) 
                = \sum_{\Bla}Q^+_{\Bla}(x;\Bt)P^-_{\Bla}(y;\Bt). 
\end{align*}

The following result gives a characterization of $P^{\pm}_{\Bla}$ and $Q^{\pm}_{\Bla}$.

\begin{thm}   %%%%  Theorem 2.14
Let $\ve \in \{ +,-\}$.  For each $\Bla \in \SP_{n,r}$, there exists a unique function 
$P^{\ve}_{\Bla}(x;\Bt)$ satisfying the following properties. 
\begin{enumerate}
\item 
$P^{\ve}_{\Bla}(x;\Bt)$ can be expressed as
\begin{equation*}
P^{\ve}_{\Bla}(x;\Bt) = \sum_{\Bmu \gve \Bla}c_{\Bla,\Bmu}(\Bt)q^{\ve}_{\Bmu}(x;\Bt),
\end{equation*}
where $c_{\Bla,\Bmu}(\Bt) \in \BQ(\Bt)$ with $c_{\Bla,\Bla}(\Bt) \ne 0$. 
\item
$P^{\ve}_{\Bla}(x;\Bt)$ can be expressed as 
\begin{equation*}
P^{\ve}_{\Bla}(x;\Bt) = \sum_{\Bmu \lve \Bla}u_{\Bla,\Bmu}(\Bt)s_{\Bmu}(x),
\end{equation*}
where $u_{\Bla,\Bmu}(\Bt) \in \BQ(\Bt)$ with $u_{\Bla,\Bla}(\Bt) = 1$. 
\end{enumerate}
A similar property holds also for $Q^{\ve}_{\Bla}(x;\Bt)$ by replacing the condition 
for $c_{\Bla,\Bmu}(\Bt), u_{\Bla,\Bmu}(\Bt)$ 
by $c_{\Bla,\Bla}(\Bt) = 1$ and $u_{\Bla,\Bla}(\Bt) \ne 0$. 
\end{thm}

\begin{proof}
For bases $u = \{ u_{\Bla} \}, v = \{ v_{\Bla} \}$ of $\Xi^n_{\BQ}(\Bt)$, we denote by 
$M = M(u,v)$ the transition matrix $(m_{\Bla,\Bmu})$ of two bases, where 
$u_{\Bla} = \sum_{\Bmu \in \SP_{n,r}}m_{\Bla,\Bmu}v_{\Bmu}$.  
Consider the bases $P^{\pm} = \{ P^{\pm}_{\Bla}\}, q^{\pm} = \{ q^{\pm}_{\Bla}\}, 
s = \{ s_{\Bla}\}, m = \{ m_{\Bla} \}$ of $\Xi^n_{\BQ}(\Bt)$, and put 
\begin{equation*}
A_{\pm} = M(q^{\pm}, P^{\pm}), \quad B_{\pm} = M(m, P^{\pm}).
\end{equation*}
We want to show that $A_{\pm}$ is upper triangular.  
By Proposition 2.12, $M(P^{\pm}, s)$ is lower unitriangular. On the other hand, 
since the total order $\lve$ is compatible with the dominance order $\le$ on $\SP_{n,r}$, 
$M(s,m)$ is lower unitriangular (the verification is reduced to the case where $r = 1$, in which
case, it is well-known).   
Thus $B_{\pm}\iv = M(P^{\pm}, s)M(s,m)$ is lower unitriangular, and $B_{\pm}$ is 
also lower unitriangular. 
Put $D_{\pm} = {}^tB_{\mp}A_{\pm}$.  
If we put $A_+ = (A^+_{\Bla, \Bmu}), B_- = (B^-_{\Bla, \Bmu})$, we have, 
by (2.8.1) and (2.13.3),

\begin{align*}
\sum_{\Bnu}q^+_{\Bnu}(x;\Bt)m_{\Bnu}(y) 
              &= \sum_{\Bmu, \Bmu', \Bnu}A^+_{\Bnu, \Bmu}B^-_{\Bnu,\Bmu'}
                           P^+_{\Bmu}(x;\Bt)P^-_{\Bmu'}(y;\Bt)  \\
              &= \sum_{\Bla}b_{\Bla}(\Bt)P^+_{\Bla}(x;\Bt)P^-_{\Bla}(y;\Bt).
\end{align*}
Since $P^+_{\Bmu}(x;\Bt)P^-_{\Bmu'}(y;\Bt)$ are linearly independent, this implies that 
$D_+ = {}^tB_-A_+$ is a diagonal matrix with $\Bla\Bla$-entry $b_{\Bla}(\Bt)$.  
Hence $A_+$ is upper triangular with $A^+_{\Bla\Bla} = b_{\Bla}(\Bt)$. 
A similar argument, by using the formula for $\sum m_{\Bnu}(x)q^-_{\Bnu}(y;\Bt)$ in (2.8.1), 
shows that $A_-$ is upper triangular with $A^-_{\Bla\Bla} = b_{\Bla}(\Bt)$. 
Thus $P^{\pm}_{\Bla}(x;\Bt)$ satisfies the conditions (i) and (ii). 
\par
Next we show the uniqueness of $P^{\pm}_{\Bla}$. 
Take $\ve \in \{ +,-\}$, and assume that $R$ satisfies the condition (i) and (ii)
for $\ve$. By (i) and (ii) for $P^{\ve}_{\Bla}$, one can write as 
\begin{align*}
s_{\Bmu}(x) &= P^{\ve}_{\Bmu}(x;\Bt) + \sum_{\Bmu' \lv \Bmu}
       u'_{\Bmu,\Bmu'}(\Bt)P^{\ve}_{\Bmu'}(x;\Bt), \\
q^{\ve}_{\Bmu}(x;\Bt) &= \sum_{\Bmu' \gve \Bmu}
       c'_{\Bmu,\Bmu'}(\Bt)P^{\ve}_{\Bmu'}(x;\Bt).
\end{align*}
It follows that 

\begin{align*}
R(x;\Bt) &= P^{\ve}_{\Bla}(x;\Bt) + \sum_{\Bmu'' \lv \Bla}u''_{\Bla,\Bmu''}(\Bt)P^{\ve}_{\Bmu''}(x;\Bt), \\
R(x;\Bt) &= \sum_{\Bmu'' \gve \Bla}c''_{\Bla,\Bmu''}(\Bt)P^{\ve}_{\Bmu''}(x;\Bt). 
\end{align*}
Hence  $R = P^{\ve}_{\Bla}$, and the uniqueness follows. 
This proves the theorem for $P^{\ve}_{\Bla}(x;\Bt)$. 
\par
The above discussion shows that $c_{\Bla,\Bla}(\Bt)$ (for $P^{\ve}_{\Bla}$) 
coincides with $b_{\Bla}(\Bt)\iv$. Thus by multiplying $b_{\Bla}(\Bt)$ on both sides of 
(i) and (ii), we obtain the corresponding formulas for $Q^{\ve}_{\Bla}(x;\Bt)$.   
The theorem is proved. 
\end{proof}

\para{2.15.}
Since $\{ P^{\ve}_{\Bla}(x;\Bt) \mid \Bla \in \SP_{n,r}\}$ and 
$\{ s_{\Bla}(x) \mid \Bla \in \SP_{n,r}\}$ are bases of 
$\Xi^n_{\BQ}(\Bt)$, there exist unique  functions
$K^{\pm}_{\Bla,\Bmu}(\Bt) \in \BQ(\Bt)$ satisfying the properties

\begin{equation*}
\tag{2.15.1}
s_{\Bla}(x) = \sum_{\Bmu \in \SP_{n,r}}K^{\pm}_{\Bla,\Bmu}(\Bt)P^{\pm}_{\Bmu}(x;\Bt).
\end{equation*}
$K^{\pm}_{\Bla,\Bmu}(\Bt) \in \BQ(\Bt) = \BQ(t_1, \dots, t_r)$ are called 
multi-variable Kostka functions.
By definition, $K^{\pm}_{\Bla,\Bmu} = 0$ unless $\Bla \gve \Bmu$. 
\par
Put $\Bt_1 = (t, \dots, t)$, and let $\SA_1$ be the $\BQ$-subalgebra of $\BQ(\Bt)$ 
consisting of rational functions $f/g$ such that $g(\Bt_1) \ne 0$. 
Put $\Xi^n_{\SA_1} = \SA_1\otimes_{\BZ}\Xi^n$.  We have 
$\Xi^n_{\SA_1}|_{\Bt = \Bt_1} = \Xi^n_{\BQ}(t)$. 
We have
\begin{equation*}
\Om(x,y; \Bt)|_{\Bt = \Bt_1} = \Om(x,y; t) 
     = \prod_{k \in \BZ/r\BZ}\prod_{i,j}
         \frac{1 - tx^{(k)}_iy^{(k+1)}_j}{1 - x^{(k)}_iy^{(k)}_j}.
\end{equation*}
Thus if we define a bilinear form $\lp \ , \ \rp_{\, 1}$ on $\Xi^n_{\BQ}(t)$ by 
\begin{equation*}
\lp q^+_{\Bla}(x;t), m_{\Bmu}(x) \rp_{\, 1} = \d_{\Bla,\Bmu},
\end{equation*}
[S1, Proposition 2.5] implies that the restriction of $\lp \ , \ \rp$ on $\Xi^n_{\SA_1}$ 
induces the form $\lp \ , \ \rp_{\, 1}$ on $\Xi^n_{\BQ}(t)$ by putting $\Bt = \Bt_1$.  
By comparing [S1, Proposition 4.8] with Proposition 2.12, we see that 
\begin{equation*}
\tag{2.15.2}
P^{\pm}_{\Bla}(x;\Bt_1) = P^{\pm}_{\Bla}(x;t), \quad
Q^{\pm}_{\Bla}(x;\Bt_1) = Q^{\pm}_{\Bla}(x;t), 
\end{equation*}
\par\medskip
\noindent
where $P^{\pm}_{\Bla}(x;t), Q^{\pm}_{\Bla}(x;t)$ are Hall-Littlewood functions 
defined in [S1, Theorem 4.4]. 
In particular, we have

\begin{equation*}
\tag{2.15.3}
K^{\pm}_{\Bla,\Bmu}(t, \dots, t) = K^{\pm}_{\Bla,\Bmu}(t),
\end{equation*}
where $K^{\pm}_{\Bla,\Bmu}(t)$ is the Kostka function (with one variable) defined 
in [S1, 5.2]. 

\remark{2.16.}
In the one-variable case, a simple algorithm of computing Kostka functions 
$K^{\pm}_{\Bla,\Bmu}(t)$ was given in [S1, Theorem 5.4] in connection with the representation 
theory of the complex reflection group $S_n\ltimes (\BZ/r\BZ)^n$. This formula is based 
on the formula (2.5.2) in [S1]. 
Since we don't have an analogous formula for the multi-parameter case, we don't know 
whether or not those Kostka functions $K^{\pm}_{\Bla,\Bmu}(\Bt)$ have a relationship with 
complex reflection groups as above.  
%%%%
%%%%
%%%%
\par\bigskip\medskip
\section{Comparison of Hall-Littlewood functions for different $r$}

\para{3.1.}
Let $\SP^a_{n,r}$ be the set of $\Bla = (\la^{(1)}, \dots, \la^{(r)}) \in \SP_{n,r}$
such that $\la^{(k)} = \emptyset$ for $k = 1, \dots, a$. 
We identify $\SP^a_{n,r}$ with $\SP_{n, r-a}$ by 
$\Bla \lra \Bla' = (\la^{(a+1)}, \dots, \la^{(r)})$.
We consider the variables $x' = (x^{(a+1)}, \dots, x^{(r)})$ and 
$\Bt' = (t'_{a+1}, \dots, t'_r)$.  Assume that $\Bla \in \SP^a_{n,r}$.  
One can consider Hall-Littlewood 
functions $P^{\pm}_{\Bla'}(x';\Bt')$ and $Q^{\pm}_{\Bla'}(x';\Bt')$ 
associated to those data.  In this section, we discuss the relationship 
between $P^{\pm}_{\Bla}(x;\Bt)$ and 
$P^{\pm}_{\Bla'}(x';\Bt')$, and also between $Q^{\pm}_{\Bla}(x;\Bt)$ and 
$Q^{\pm}_{\Bla'}(x';\Bt')$. 
\par  
First consider the case where $a = 1$, and put $x' = (x^{(2)}, \dots, x^{(r)})$.
We denote by ${q'}^{(k)}_{s,\pm}(x';t)$ the function of $x'$ corresponding to the 
function $q^{(k)}_{s,\pm}(x;t)$ of $x$ (note that $k \in \BZ/(r-1)\BZ$ for 
${q'}^{(k)}_{s,\pm}$).   
We have a lemma. 

\begin{lem}  %%%%  Lemma 3.2.

\begin{enumerate}
\item
${q'}^{(k)}_{s,+}(x';t_{k-1}) = q^{(k)}_{s, +}(x;t_{k-1})$ for $k = 3, \dots, r$, and  
\begin{equation*}
\tag{3.2.1}
{q'}^{(2)}_{s,+}(x'; t_1t_r) = \sum_{\substack{s',s'' \\ s' + s'' = s}}
       (t_1)^{s''}q^{(2)}_{s',+}(x;t_1)q^{(1)}_{s'',+}(x; t_r).
\end{equation*}

\item
${q'}^{(k)}_{s,-}(x';t_k) = q^{(k)}_{s,-}(x;t_k)$ for $k = 2, \dots, r-1$, and 
\begin{equation*}
\tag{3.2.2}
{q'}^{(r)}_{s,-}(x'; t_1t_r) = \sum_{\substack{s',s'' \\ s' + s'' = s}}
       (t_r)^{s''}q^{(r)}_{s',-}(x;t_r)q^{(1)}_{s'',-}(x; t_1).
\end{equation*}
\end{enumerate}
\end{lem}

\begin{proof}
We prove (i).  The first statement is clear from the definition.  Recall, 
by (2.5.2), that 
\begin{equation*}
\prod_{i = 1}^m\frac{1 - tux_i^{(k-1)}}{1 - ux_i^{(k)}} 
= \sum_{s \ge 0}q^{(k)}_{s,+}(x;t)u^s.
\end{equation*}
Since 
\begin{equation*}
\prod_{i = 1}^m\frac{1 - t_1ux_i^{(1)}}{1 - ux_i^{(2)}}
      \prod_{i=1}^m\frac{1 - t_rt_1ux_i^{(r)}}{1 - t_1ux_i^{(1)}}
         = \prod_{i=1}^m\frac{1 - t_1t_rux_i^{(r)}}{1 - ux_i^{(2)}},
\end{equation*}
we have 
\begin{equation*}
\biggl(\sum_{s' \ge 0}q^{(2)}_{s',+}(x;t_1)u^{s'}\biggr)
   \biggl(\sum_{s'' \ge 0}q^{(1)}_{s'',+}(x; t_r)(t_1u)^{s''}\biggr)
    = \sum_{s \ge 0}{q'}^{(2)}_{s,+}(x';t_1t_r)u^s.
\end{equation*}
Thus (i) holds.  The proof for (ii) is similar. 
\end{proof}

\para{3.3.}
By the map $\Bla' = (\la^{(2)}, \dots, \la^{(r)}) \mapsto 
\Bla = (-, \la^{(2)}, \dots, \la^{(r)})$,
we can identify 
$\SP_{n,r-1}$ with the subset $\SP^1_{n,r}$ of $\SP_{n,r}$.
Since the dominance order on $\SP_{n,r-1}$ is compatible with the dominance order of 
$\SP_{n,r}$, one can choose a total order on $\SP_{n,r}$ compatible with the total 
order on $\SP_{n,r-1}$, namely, which satisfies the property 
if $\Bla \in \SP^1_{n,r}$ and $\Bmu \lve \Bla$, then $\Bmu \in \SP^1_{n,r}$.  
More generally, by considering the sequence 
$\SP_{n, 1} \subset \SP_{n,2}\subset \cdots \subset \SP_{n,r}$,  
we can choose a total order on $\SP_{n,r}$
so that it is compatible with the total order on each subset $\SP_{n,k}$.   
\par
Assume that $\Bla \in \SP_{n,r}^1$. Then $\Bmu \in \SP_{n,r}^1$ for any $\Bmu \lve \Bla$,  
and $s_{\Bmu}(x) = s_{\Bmu'}(x')$ is a function with respect 
to the variable $x' = (x^{(2)}, \dots, x^{(r)})$. 
Since $Q^{\pm}_{\Bla}$ is a linear combination of $s_{\Bmu}$ with $\Bmu \in \SP_{n,r}^1$,
$Q^{\pm}_{\Bla}$ is a function with respect to the variable $x'$. We show  the following.

\begin{prop}  %%%%  Prop. 3.4
Assume that $\Bla \in \SP_{n,r}^1$.  Let ${Q'}^{\pm}_{\Bla'}$ be the 
function defined with respect to $x' = (x^{(2)}, \dots, x^{(r)})$, and 
$\Bla' \in \SP_{n,r-1}$, the element corresponding to $\Bla$.  Then 
\begin{align*}
Q^{\pm}_{\Bla}(x;t_1, \dots, t_r) = {Q'}^{\pm}_{\Bla'}(x'; t_2, \dots, t_{r-1}, t_1t_r). \\
\end{align*}
\end{prop}  

\begin{proof}
Since ${Q'}^{\pm}_{\Bla'}$ is written as a linear combination of $s_{\Bmu'}(x')$ 
with $\Bmu' \lve \Bla'$, by our choice of the total order $\lve$ on $\SP_{n,r}$,  
 it is written as a linear combination of $s_{\Bmu}(x)$ with $\Bmu \lve \Bla$. 
Thus it is enough to show  that ${Q'}^{\pm}_{\Bla'}$ can be written as 
a linear combination of $q^{\pm}_{\Bmu}$ with $\Bmu \gve \Bla$ such that the coefficient of 
$q^{\pm}_{\Bla}$ is equal to 1.  We can write 
\begin{equation*}
{Q'}^{\pm}_{\Bla'}(x';\Bt') = {q'}^{\pm}_{\Bla'}(x';\Bt') + 
      \sum_{\substack{\Bmu' \in\SP_{n,r-1} \\ \Bmu' \gv \Bla'}}
                 c_{\Bla',\Bmu'}(\Bt'){q'}^{\pm}_{\Bmu'}(x';\Bt'),
\end{equation*} 
where $\Bt' = (t_2, \dots, t_{r-1}, t_1t_r)$. 
Here for $\Bmu' = (\mu^{(2)}, \dots, \mu^{(r)})$,
\begin{equation*}
{q'}^+_{\Bmu'}(x';\Bt') = \prod_{i = 1}^m{q'}^{(2)}_{\mu^{(2)}_i,+}(x';t_1t_r)
              \prod_{k = 3}^r\prod_{i = 1}^m q^{(k)}_{\mu^{(k)}_i,+}(x; t_{k-1}).
\end{equation*} 
${q'}^{(2)}_{\mu^{(2)}_i,+}(x';t_1t_r)$ can be written as a linear combination of 
$q^{(2)}_{s',+}(x;t_1)q^{(1)}_{s'',+}(x;t_r)$ by (3.2.1), where $s', s''\ge 0$ are such that 
$\mu^{(2)}_i = s' + s''$.   Hence ${q'}^+_{\Bmu'}$ can be written as a linear combination of 
various $q^+_{\Bnu}$, where $\Bnu = (\nu^{(1)}, \dots, \nu^{(r)})$ are $r$-compositions such that
$\nu^{(1)}_i + \nu^{(2)}_i = \mu^{(2)}_i$ for each $i$ and that $\nu^{(k)} = \mu^{(k)}$ for 
$k \ge 3$.   Then clearly $\Bnu \ge \Bmu$ for $\Bmu = (-,\mu^{(2)}, \dots, \mu^{(r)})$. 
If we denote by $\wt\Bnu$ the $r$-partition obtained 
from $\Bnu$ by rearranging the order, then we have $\wt \Bnu \ge \Bnu$.
This is true also for ${q'}^+_{\Bla'}$.    
It follows that ${Q'}^+_{\Bla'}$ is a linear combination of various $q^+_{\Bmu}$ for 
$\Bmu \in \SP_{n,r}$ such that $\Bmu \gve \Bla$.  The term $q^+_{\Bla}$ comes only from 
${q'}^+_{\Bla'}$, and it is easily checked that the coefficient of $q^+_{\Bla}$ is equal to 1. 
This proves the proposition in the ``$+$''case. The proof for the ``$-$''case is 
done similarly by using (3.2.2).
\end{proof}

As a corollary, we have the following result, which describes the relationship
of Hall-Littlewood functions and Kostka functions for different $r$. 

%%%%
%%%%
\begin{thm}   %%%%   Theorem  3.5
\begin{enumerate}
\item
Assume that $\Bla \in \SP^a_{n,r}$ for $1 \le a < r$.  
Let ${Q'}^{\pm}_{\Bla'}, {P'}^{\pm}_{\Bla'}$ be the functions 
defined with respect to 
$x' = (x^{(a+1)}, \dots, x^{(r)})$, and $\Bla' \in \SP_{n, r-a}$ 
the element corresponding to $\Bla$. 
Then
\begin{align*}
Q^{\pm}_{\Bla}(x; t_1, \dots, t_r) &= 
   {Q'}^{\pm}_{\Bla'}(x', t_{a+1}, t_{a+2},\dots, t_{r-1}, t_at_{a-1}\cdots t_1t_r),  \\
P^{\pm}_{\Bla}(x; t_1, \dots, t_r) &= 
   {P'}^{\pm}_{\Bla'}(x', t_{a+1}, t_{a+2},\dots, t_{r-1}, t_at_{a-1}\cdots t_1t_r).
\end{align*}   
\item 
Let $\Bla, \Bmu \in \SP_{n,r}$ be such that $\Bmu \lve \Bla$.  Assume that 
$\Bla \in \SP_{n,r}^{a}$.  Then $\Bmu \in \SP_{n,r}^{a}$, 
and we have
\begin{equation*}
K^{\pm}_{\Bla,\Bmu}(t_1, \dots, t_r) = 
     {K'}^{\pm}_{\Bla',\Bmu'}(t_{a+1}, \dots, t_{r-1}, t_at_{a-1}\cdots t_1t_r),
\end{equation*}
where ${K'}^{\pm}_{\Bla',\Bmu'}$ is the Kostka function associated to 
$\Bla', \Bmu' \in \SP_{n, r-a}$. 
\item 
Assume that $a = r-1$. For $\Bla = (-,\dots, -,\la^{(r)}) \in \SP_{n,r}^{r-1}$, 
by setting $t_1 = \cdots = t_r = t$, we have
\begin{align*}
Q^{\pm}_{\Bla}(x; t) &= Q_{\la^{(r)}}(x^{(r)}, t^r),  \\ 
P^{\pm}_{\Bla}(x; t) &= P_{\la^{(r)}}(x^{(r)}, t^r),  \\ 
\end{align*}
where the left hand side is the one variable Hall-Littlewood functions associated to 
the $r$-partition $\Bla$ $($see $(2.15.2))$, 
and the right hand side is the classical Hall-Littlewood functions 
associated to the partition $\la^{(r)}$. 
\item
Under the same assumption as in $({\rm iii})$,  take $\Bmu \in \SP_{n,r}$
such that $\Bmu \lve \Bla$.  Then $\Bmu = (-, \dots, -, \mu^{(r)}) \in \SP_{n,r}^{r-1}$, 
and we have
\begin{equation*}
K^{\pm}_{\Bla,\Bmu}(t) = K_{\la^{(r)},\mu^{(r)}}(t^r),
\end{equation*}
where the left hand side is the one-variable Kostka function  
associated to $r$-partitions $($see $(2.15.3))$, and the right hand side is 
the classical Kostka polynomial associated to partitions. 
\end{enumerate}
\end{thm}

\begin{proof}
The first formula of (i) follows from Proposition 3.4. 
In this formula, $Q^{\pm}_{\Bla}(x;\Bt)$ has an expansion in terms of Schur
functions
\begin{equation*}
\tag{3.5.1}
Q^{\pm}_{\Bla}(x;\Bt) = \sum_{\Bmu \lve \Bla}u_{\Bla,\Bmu}(\Bt)s_{\Bmu}(x).
\end{equation*}
On the other hand, if $\Bmu \lve \Bla$ and $\Bla \in \SP^a_{n,r}$, 
then $\Bmu \in \SP^a_{n,r}$, and $s_{\Bmu}(x) = s_{\Bmu'}(x')$. 
Thus (3.5.1) also gives an expansion of ${Q'}^{\pm}_{\Bla'}(x';\Bt')$ in terms of 
Schur functions for $\Bmu' \in \SP_{n, r-a}$,  
\begin{equation*}
\tag{3.5.2}
{Q'}^{\pm}_{\Bla'}(x';\Bt') = \sum_{\Bmu' \lve \Bla'}u'_{\Bla',\Bmu'}(\Bt')s_{\Bmu'}(x')
\end{equation*} 
with $u_{\Bla,\Bla}(\Bt) = u'_{\Bla',\Bla'}(\Bt')$, where 
$\Bt' = (t_{a+1}, t_{a+2}, \dots, t_{r-1}, t_at_{a-1}\cdots t_1t_r)$.
By Theorem 2.14, we have $P^{\pm}_{\Bla}(x;\Bt) = u_{\Bla,\Bla}(\Bt)\iv Q^{\pm}_{\Bla}(x;\Bt)$, 
${P'}^{\pm}_{\Bla'}(x';\Bt') = {u'}_{\Bla',\Bla'}(\Bt')\iv {Q'}^{\pm}_{\Bla'}(x';\Bt')$. 
Thus we obtain the second formula of (i). 
Now (ii) is immediate from the second formula of (i). 
(iii) and (iv) are the special case of (i) and (ii).  
\end{proof}

\remark{3.6.}  In the case where $r = 2$, the formula (iv) was first 
proved by Achar-Henderson in [AH, Corollary 5.3] by a geometric method. 
After that a combinatorial proof of (iv) and the related formula (iii) 
for Hall-Littlewood functions (for $r = 2$) were given 
in [LS, Proposition 1.11 and Corollary 1.12].
This argument also works for the general $r$.   
In those discussions, the proof proceeds under the one-variable setting, 
namely under the setting where $t_1 = t_2 = \cdots = t_r = t$.  
However, in order to describe the relationship among Kostka functions and 
Hall-Littlewood functions as in the setting of (i) and (ii), one needs to introduce 
multi-variable Kostka functions and Hall-Littlewood functions. 
Note that the proof of (iii) and (iv) here is much simpler than 
the discussion in [LS]. 
%%%
%%%
%%%
\par\medskip
\section{ Closed formula for Hall-Littlewood functions}

\para{4.1.}
In this section, we define a function $R^{\pm}_{\Bla}(x;\Bt)$, as 
an analogue of the function $R_{\la}(x;t)$ in [M, III, 1], and show 
in this and next section that $Q^{\pm}_{\Bla}(x;\Bt)$ has an explicit 
description in terms of $R^{\pm}_{\Bla}(x;\Bt)$ under a mild restriction. 
\par  
Let $\SM$ be as in (2.4.1). 
We define a total order on $\SM$ by 
\begin{equation*}
  (1,1) < (2,1) < \cdots < (r,1) < (1,2) < (2,2) < \cdots < (r,2) < \cdots. 
\end{equation*}
We fix $\Bla = (\la^{(1)}, \dots, \la^{(r)}) \in \SP_{n,r}$ with 
$\la^{(k)} = (\la^{(k)}_1, \dots, \la^{(k)}_m)$ for a common $m \ge 1$. 
Write $\la^{(k)}_i = \la_{\nu}$ and $x^{(k)}_i = x_{\nu}$ if $\nu = (k,i) \in \SM$. 
Let $\nu_0 = (k_0,i_0) \in \SM$ be the largest element such that 
$\la_{\nu_0} \ne 0$. We put $b(\nu) = k$ if $\nu = (k,i)$. 
We define a function $I^{\pm}_{\nu}(x;\Bt)$ for $\nu = (k,i)  \in \SM$ by 
\begin{equation*}
\tag{4.1.1}
I^{\pm}_{\nu}(x;\Bt) = \begin{cases}
             \displaystyle\prod_{(k,i) <  (k \mp 1, j) }
                      (x^{(k)}_i - t_{k-c}x^{(k\mp 1)}_j)
                       &\quad\text{ if $(k,i) \le \nu_0$ and $\la^{(k)}_i \ne 0$, } \\  \\
             \hspace{6mm}\displaystyle\prod_{i < j}
                      (x^{(k)}_i - t_{k-c}x^{(k\mp 1)}_j)
                       &\quad\text{ if $(k,i) \le \nu_0$ and $\la^{(k)}_i =  0$, } \\  \\
             \hspace{6mm}\displaystyle\prod_{\substack{ i < j}}
                      (x^{(k)}_i - t_kx^{(k)}_j)
                        &\quad\text{ if $(k,i) > \nu_0$, }
                      \end{cases}
\end{equation*}
where $c$ is as in (2.5.3). 
We regard $k \in \BZ/r\BZ$.
\par
Let $S_m^r = S_m \times \cdots \times S_m$ ($r$-factors) 
be the permutation group of the variables 
$x = (x^{(1)}, \dots, x^{(r)})$. 
We define a function $R_{\Bla}^{\pm}(x;\Bt)$ by 
\begin{equation*}
\tag{4.1.2}
R^{\pm}_{\Bla}(x;\Bt) = \sum_{w \in S_m^r}w\left(
               \frac{\displaystyle\prod_{1 \le k \le r}(x^{(k)})^{\la^{(k)} - \ve_{\pm}^{(k)}}
                        \prod_{\nu \in \SM }I^{\pm}_{\nu}(x;\Bt)}
                       {\displaystyle\prod_{1 \le k \le r}\prod_{i < j}x_i^{(k)} - x_j^{(k)}}\right),     
\end{equation*}
where $\ve^{(k)}_{\pm} = (\ve^{(k)}_{1,\pm} \dots, \ve^{(k)}_{m, \pm}) \in \BZ^m$ 
is given by  
\begin{align*}
\ve^{(k)}_{i,+} &= \begin{cases}
                 1  &\quad\text{ if  $\la^{(k)}_i \ne 0$ and $k = 1$, }\\
                 0  &\quad\text{ otherwise,}
              \end{cases}  \\
\ve^{(k)}_{i,-} &= \begin{cases}
                 1  &\quad\text{ if  $\la^{(k)}_i \ne 0$ and $k \ne r$, }\\
                 0  &\quad\text{ otherwise. } 
                   \end{cases}
\end{align*}
It follows from the definition that $R^{\pm}_{\Bla}(x;\Bt) \in \BZ[x_{\SM};\Bt]$ and that
\begin{equation*}
\tag{4.1.3}
R^{\pm}_{\Bla}(x;\bold 0) = s_{\Bla}(x).
\end{equation*}

\par
For $\Bla \in \SP_{n,r}$, define a subgroup $S'_{\Bla}$ of $S_m^r$ by 
$S'_{\Bla} = S_{\la'_1} \times \cdots \times S_{\la'_r}$,
where $\la'_k = \sharp\{ 1 \le i \le m \mid (k,i) > \nu_0 \}$.
We regard $S_{\la'_k}$ as the permutation group of the  set 
$\{ x^{(k)}_i \mid (k,i) > \nu_0 \}$ fixing any variable $x^{(k)}_i$ for $(k,i) \le \nu_0$.
We define a polynomial $v'_{\Bla}(\Bt) \in \BZ[\Bt]$ by 
\begin{equation*}
\tag{4.1.4}
v'_{\Bla}(\Bt) = \prod_{k=1}^rv_{\la'_k}(t_k),
\end{equation*}
where $v_i(t)$ is given as in (2.2.1). 

\para{4.2.}
We consider the special case where $\Bla = ((s); -; \cdots;-)$ with $s \ge 1$. 
Then $S'_{\Bla} = S_{m-1} \times S^{r-1}_m$, where 
$S_{m-1}$ is the stabilizer of $x^{(1)}_1$. 
Hence $v'_{\Bla}(\Bt) = v_{m -1}(t_1)\prod_{k \ge 2}v_m(t_k)$. 
We have
\begin{equation*}
\begin{split}
\prod_k&(x^{(k)})^{\la^{(k)} - \ve^{(k)}_{\pm}}\prod_{\nu \in \SM} I_{\nu}^{\pm}(x;\Bt)  \\ 
          &= (x^{(1)}_1)^{s-1}\prod_{j \ge 1}(x^{(1)}_1 - t_{1-c}x^{(1\mp 1)}_j)
                \prod_{2 \le i < j}(x^{(1)}_i - t_1x^{(1)}_j)
                   \prod_{k \ge 2, i < j}(x^{(k)}_i - t_kx^{(k)}_j).
\end{split}
\end{equation*}
$S'_{\Bla}$ stabilizes the factor
\begin{equation*}
            (x^{(1)}_1)^{s-1}\prod_{j \ge 1}(x^{(1)}_1 - t_{1-c}x^{(1 \mp 1)}_j)/
                \prod_{1 < j}(x_1^{(1)} - x_j^{(1)}) 
\end{equation*}
and by [M, p.207, (1.4)]
\begin{equation*}
\sum_{w' \in S'_{\Bla}}w'\biggl(\prod_{2 \le i < j}
             \frac{x^{(1)}_i - t_1x^{(1)}_j}{x^{(1)}_i - x^{(1)}_j}
                        \prod_{k \ge 2, i < j}
             \frac{x^{(k)}_i - t_kx^{(k)}_j}{x^{(k)}_i - x^{(k)}_j}\biggr)
        = v_{m-1}(t_1)\prod_{k \ge 2}v_m(t_k) = v'_{\Bla}(\Bt).
\end{equation*}
It follows that
\begin{align*}
R_{\Bla}^{\pm}(x;\Bt) &= 
      v'_{\Bla}(\Bt)\sum_{w \in S_m/S_{m-1}}w\biggl(
              (x_1^{(1)})^{s-1}(x_1^{(1)} - t_{1-c}x_1^{(1 \mp 1)})
       \prod_{1 < j}\frac{x_1^{(1)} - t_{1-c}x_j^{(1 \mp 1)}}{x_1^{(1)} - x_j^{(1)}}\biggr)  \\
             &= v'_{\Bla}(\Bt)\sum_{i = 1}^m
               (x_i^{(1)})^{s-1}(x_i^{(1)} - t_{1-c}x_i^{(1 \mp 1)})
       \prod_{j \ne i}\frac{x_i^{(1)} - t_{1-c}x_j^{(1 \mp 1)}}{x_i^{(1)} - x_j^{(1)}}  \\
            &= v'_{\Bla}(\Bt)q_{s,\pm}^{(1)}(x;\Bt). 
\end{align*}
Hence we have 
\par\medskip
\noindent 
(4.2.1) \ Assume that $\Bla = ((s); -; \cdots; -)$.  Then 
$R^{\pm}_{\Bla}(x;\Bt) = v'_{\Bla}(\Bt)q^{\pm}_{\Bla}(x;\Bt)$.

\par\medskip
The above computation can be generalized as follows. 
For $\Bla \in \SP_{n,r}$, 
$S'_{\Bla}$ stabilizes the expression 
\begin{equation*}
\prod_{1 \le k \le r}(x^{(k)})^{\la^{(k)} - \ve^{(k)}_{\pm}}
      \prod_{\nu \le \nu_0 }I^{\pm}_{\nu}(x;\Bt).
\end{equation*}
Moreover, by [M, p.207, (1.4)],  
\begin{equation*}
\sum_{w' \in S'_{\Bla}}w'\biggl( \prod_{(k,i) > \nu_0}\prod_{\substack{i < j}}
           \frac{x^{(k)}_i - t_kx^{(k)}_j}{x^{(k)}_i - x^{(k)}_j}\biggr) = v'_{\Bla}(\Bt). 
\end{equation*}
Thus we have an expression for $R^{\pm}_{\Bla}(x;\Bt)$,
\begin{equation*}
\tag{4.2.2}
\begin{split}
   R^{\pm}_{\Bla}(x;\Bt)   
     = v'_{\Bla}(\Bt)\sum_{w \in S^r_m/S'_{\Bla}}
    w\left(\frac{\displaystyle\prod_{1 \le k \le r}(x^{(k)})^{\la^{(k)} - \ve_{\pm}^{(k)}}
                \prod_{(k,i) \le \nu_0}\prod_{(*) }
                          x^{(k)}_i - t_{k-c}x^{(k\mp 1)}_j}
                 {\displaystyle\prod_{(k,i) \le \nu_0}\prod_{\substack{i < j }}
                          x^{(k)}_i - x^{(k)}_j}  \right),      
\end{split}
\end{equation*}
where (*) is given by the first and the second condition in (4.1.1). 
It follows from (4.2.2) that
\par\medskip\noindent
(4.2.3)  $R^{\pm}_{\Bla}(x;\Bt) \in \BZ[x_{\SM};\Bt]$ is divisible by $v'_{\Bla}(\Bt)$. 
$v'_{\Bla}(\Bt)\iv R^{\pm}_{\Bla}(x;\Bt) \in \BZ[x_{\SM};\Bt]$ 
has the stability for the increase of variables 
$\{x^{(k)}_1, \dots, x^{(k)}_m\}$  to $\{ x_1^{(k)}, \dots, x^{(k)}_{m+1}\}$. In particular, 
by taking $m \mapsto \infty$, we obtain $v'_{\Bla}(\Bt)\iv R^{\pm}_{\Bla} \in \Xi^n[\Bt]$.  
\par\medskip
We define a polynomial $\wt P^{\pm}_{\Bla}(x;\Bt) \in \BZ[x_{\SM};\Bt]$ by 
\begin{equation*}
\tag{4.2.4}
\wt P^{\pm}_{\Bla}(x;\Bt) = v'_{\Bla}(\Bt)\iv R^{\pm}_{\Bla}(x;\Bt). 
\end{equation*}
(4.2.4) is an analogue of [M, III, (2.1)] in the classical case. But note 
that $v'_{\Bla}(\Bt)$ is different from $v_{\la}(t)$ there. 
Let $p$ be the largest number such that $\la^{(k)}_p \ne 0$ for 
$k = 1, \dots, r-1$, and put 
$j_0 = \max\{\ell(\la^{(r)}) -p, 0\}$. 
We define a polynomial $\wt Q^{\pm}_{\Bla}(x;\Bt) \in \BZ[x_{\SM};\Bt]$ by 
\begin{equation*}
\tag{4.2.5}
\wt Q^{\pm}_{\Bla}(x;\Bt) = 
                (1 - t_0)^{j_0}v'_{\Bla}(\Bt)\iv R^{\pm}_{\Bla}(x;\Bt)
                          =  (1 - t_0)^{j_0}\wt P^{\pm}_{\Bla}(x;\Bt),
\end{equation*} 
where we put $t_0 = t_1\cdots t_r$. 
By taking $m \to \infty$, $\wt P^{\pm}_{\Bla}(x;\Bt), \wt Q^{\pm}_{\Bla}(x;\Bt)$ 
determine symmetric functions $\wt P^{\pm}_{\Bla}, \wt Q^{\pm}_{\Bla} \in \Xi^n[\Bt]$.  
\par\medskip
Next result describes the expansion of $R^{\pm}_{\Bla}(x;\Bt)$
in terms of Schur functions.  Note that in this formula, the total 
order $\lve$ is replaced by the partial order $\le$.  

\begin{prop} %%%  Prop. 4.3
For a given $\Bla \in \SP_{n,r}$, fix a set $\SM$, and 
consider $R^{\pm}_{\Bla}(x;\Bt) \in \BZ[x_{\SM};\Bt]$. 
Then there exist polynomials $u^{\pm}_{\Bla,\Bmu}(\Bt) \in \BZ[\Bt]$ such that
\begin{equation*}
R^{\pm}_{\Bla}(x;\Bt) = \sum_{\Bmu \le \Bla}u^{\pm}_{\Bla,\Bmu}(\Bt)s_{\Bmu}(x). 
\end{equation*}
Moreover, $u^{\pm}_{\Bla\Bmu}(\bold 0) = 0$ if $\Bmu \ne \Bla$, and  
$u^{\pm}_{\Bla\Bla}(\bold 0) = 1$. 
\end{prop}

\begin{proof}
The product $\prod_{\nu \in \SM }I^{\pm}_{\nu}(x;\Bt)$ can be written as a sum of monomials
\begin{equation*}
\tag{4.3.1}
          \prod_{\nu,\nu'}(x_{\nu})^{r_{\nu\nu'}}(-t_{\nu'}x_{\nu'})^{r_{\nu'\nu}},
\end{equation*}
where 
\begin{equation*}
t_{\nu'} = \begin{cases}
            t_{b(\nu') -1 + c}  &\quad\text{ if } \nu \le \nu_0, \\
            t_{b(\nu')}          &\quad\text{ if } \nu > \nu_0.
           \end{cases}
\end{equation*}
(See 4.1 for the definition of $b(\nu)$.)
Moreover, $(r_{\nu,\nu'})$ is an integral matrix indexed by $\SM$ consisting of 0 and 1 satisfying 
the condition

\begin{equation*}
\tag{4.3.2}
r_{\nu,\nu'} + r_{\nu',\nu} = \begin{cases}
                   1  &\quad\text{ if } \nu \le \nu_0, \nu = (k,i), \nu' = (k\mp 1, j), 
                                       \la_{\nu} \ne 0, \nu < \nu',  \\ 
                   1  &\quad\text{ if } \nu \le \nu_0, \nu = (k,i), \nu' = (k \mp 1, j), 
                                       \la_{\nu} = 0, i < j, \\
                   1  &\quad\text{ if } \nu > \nu_0, \nu = (k,i), \nu' = (k,j),  i < j, \\
                   0    &\quad\text{ otherwise,}
                              \end{cases}
\end{equation*}
and the matrices $(r_{\nu,\nu'})$ satisfying the above condition are in 1-1 correspondence 
with the above monomials.
For a given matrix $(r_{\nu,\nu'})$ as above, we put
\begin{equation*}
\b^{(k)}_i = \la^{(k)}_i - \ve^{(k)}_{i,\pm} + \sum_{\nu' \in \SM}r_{\nu,\nu'}
\end{equation*}
for $\nu = (k,i)$.  Let $\Bbe = (\b^{(1)}, \dots, \b^{(r)})$ be the $r$-composition. 
Then $\Bbe$ produces the ``Schur function'' $a_{\Bbe}/a_{\Bdel}$, where 
$a_{\Bbe} = \sum_{w \in S_m^r}\ve(w)w(x^{\Bbe})$, and $\Bdel = (\d^{(1)}, \dots, \d^{(r)})$ 
with $\d^{(k)} = (m-1, \dots, 1,0)$.  By (4.1.2), $R^{\pm}_{\Bla}$
is a sum of $\prod_{\nu'}(-t_{\nu'})^{r_{\nu',\nu}}a_{\Bbe}/a_{\Bdel}$ 
obtained from the matrix $(r_{\nu,\nu'})$. 
If the entries of the composition $\b^{(k)}$ are 
not all distinct for some $k$, then $a_{\Bbe}/a_{\Bdel} = 0$. 
So we may assume that all the entries of $\b^{(k)}$ are distinct for any $k$. 
By rearranging its entries in the decreasing order, one can write it as
\begin{equation*}
\b^{(k)}_{w_k(i)} = \mu^{(k)}_i + (m - i)
\end{equation*}
for some $w_k \in S_m$.  Then $\Bmu = (\mu^{(k)}_i) \in \SP_{n,r}$, and 
$a_{\Bbe}/a_{\Bdel}$ coincides with $\ve(w)s_{\Bmu}(x)$ for 
$w = (w_1, \dots, w_r) \in S^r_m$.
We shall show that 
\begin{equation*}
\tag{4.3.3}
\Bmu \le \Bla.
\end{equation*} 
 Define the matrix $(s_{\nu, \nu'})$ by $s_{\nu,\nu'} = r_{w(\nu), w(\nu')}$, 
where $w(\nu) = (k,w_k(i))$ for $\nu = (k,i) \in \SM$. 
One can write as
\begin{equation*}
\tag{4.3.4}
\mu^{(k)}_i + (m - i) = \la^{(k)}_{w_k(i)} - \ve^{(k)}_{w_k(i),\pm}
                + \sum_{\nu' \in \SM}s_{\nu,\nu'},
\end{equation*}
where $\nu = (k,i)$. 
We want to show that
\begin{equation*}
\tag{4.3.5}
\sum_{k=1}^r\sum_{i = 1}^{a}\mu^{(k)}_i + \sum_{k = 1}^s\mu^{(k)}_{a+1}
          \le \sum_{k=1}^r\sum_{i=1}^a\la^{(k)}_{w_k(i)} + \sum_{k=1}^s\la^{(k)}_{w_k(a+1)}
\end{equation*}
for $0 \le a \le m -1$ and $1 \le s \le r$. 
Note that (4.3.5) implies (4.3.3) since $w(\Bla) \le \Bla$ for any $w \in S^r_m$. 
By (4.3.4), we have

\begin{align*}
\sum_{k= 1}^r\sum_{i = 1}^a\mu^{(k)}_i &+ \sum_{k = 1}^s\mu^{(k)}_{a+1} \\
   &= \sum_{k=1}^r\sum_{i = 1}^a(\la^{(k)}_{w_k(i)} - \ve^{(k)}_{w_k(i),\pm})
               + \sum_{k = 1}^s(\la^{(k)}_{w_k(a+1)} - \ve^{(k)}_{w_k(a+1),\pm})  \\
   & - \sum_{k=1}^r\sum_{i = 1}^a(m - i) - \sum_{k = 1}^s(m - (a+1)) \\
   & + \sum_{\nu \in \SB, \nu' \in \SM}s_{\nu,\nu'},
\end{align*}

where 
\begin{equation*}
\SB = \{ (k,i) \in \SM \mid 1 \le k \le r, 1 \le i \le a\} 
           \cup \{ (k, a+1) \mid 1 \le k \le s\}.  
\end{equation*}

Hence in order to show (4.3.5), it is enough to see that
\begin{align*}
\tag{4.3.6}
\sum_{\nu \in \SB, \nu' \in \SM}s_{\nu, \nu'}
              &\le \sum_{k = 1}^r\sum_{i = 1}^a(m - i + \ve^{(k)}_{w_k(i),\pm}) 
               + \sum_{k = 1}^s(m - (a+1) + \ve^{(k)}_{w_k(a + 1),\pm}) \\
              &= \sum_{k=1}^s\sum_{i=1}^{a+1}(m - i + \ve^{(k)}_{w_k(i),\pm}) 
               + \sum_{k = s+1}^r\sum_{i = 1}^a(m - i + \ve^{(k)}_{w_k(i),\pm}).
\end{align*}
\par
Let 
$w(\SB) = \{ (k, w_k(i)) \mid (k,i) \in \SB\}$. 
Put $w'_+ = (w_2, w_3, \dots, w_r, w_1)$ and   
$w'_- = (w_r, w_1, \dots, w_{r-1})$.  Define $\SB'_{\pm} \subset \SM$ by 
$w(\SB'_{\pm}) = w'_{\pm}(\SB)$. We have

\begin{equation*}
\sum_{\nu \in \SB, \nu' \in \SM}s_{\nu,\nu'}
      = \sum_{\nu \in \SB, \nu' \in \SB'_{\pm}}s_{\nu,\nu'}
           + \sum_{\nu \in \SB, \nu' \in \SM - \SB'_{\pm}}s_{\nu,\nu'}.
\end{equation*}
For $k = 1, \dots, r$, let $X^{\pm}_k = X_k' \cup X_k''$ with
\begin{align*}
X'_k &= \{ (\nu,\nu') \in \SM^2 \mid w(\nu) \le \nu_0, b(\nu) = k, b(\nu') = k \mp 1 \}, \\
X''_k &= \{(\nu,\nu') \in \SM^2 \mid w(\nu) > \nu_0, w(\nu') > \nu_0, b(\nu) = b(\nu') = k\}, 
\end{align*}
and put
\begin{align*}
A^{\pm}_k &= \sum_{\substack{(\nu, \nu') \in X^{\pm}_k \\ \nu \in \SB, \nu' \in \SB_{\pm}'}}
              s_{\nu,\nu'},   \\
B^{\pm}_k &= \sum_{\substack{(\nu,\nu') \in X^{\pm}_k \\ 
                  \nu \in \SB, \nu' \in \SM - \SB_{\pm}'}}s_{\nu,\nu'}.
\end{align*}
By (4.3.2) we have $\sum_{k=1}^r(A_k^{\pm} + B_k^{\pm}) = 
      \sum_{\nu \in \SB, \nu' \in \SM }s_{\nu,\nu'}$. 
By our choice of $\SB$ and $\SB'_{\pm}$, the $k$-th row of $w(\SB)$ consists of 
$\{ (k, w_k(i)) \mid (k,i) \in \SB\}$, and $(k\mp 1)$-row of $w(\SB'_{\pm})$ consists of 
$\{(k\mp 1, w_k(i)) \mid (k,i) \in \SB \}$.   
Assume that $k > s$.  
It is easy to see that

\begin{align*}
\tag{4.3.7}
A^{\pm}_k &\le  a(a+1)/2 - \sharp\{ 1 \le i \le a  \mid \ve^{(k)}_{w_k(i),\pm} = 0 \}, \\
B^{\pm}_k &\le a(m - a).
\end{align*}  
Hence 
\begin{align*}
\tag{4.3.8}
A^{\pm}_k + B^{\pm}_k &\le ma - a(a+1)/2 + \sum_{i = 1}^a \ve^{(k)}_{w_k(i),\pm} \\
              &= \sum_{i=1}^a(m - i + \ve^{(k)}_{w_k(i),\pm}).
\end{align*}
A similar formula as (4.3.8) holds for the case where $k \le s$ by replacing 
$a$ by $a+1$. 
By summing up those formulas, we have
\begin{equation*}
\sum_{k=1}^r(A_k^{\pm} + B_k^{\pm}) \le 
              \sum_{k=1}^s\sum_{i=1}^{a+1}(m - i + \ve^{(k)}_{w_k(i),\pm}) 
               + \sum_{k = s+1}^r\sum_{i = 1}^a(m - i + \ve^{(k)}_{w_k(i),\pm}).
\end{equation*}
Thus (4.3.6) holds, and so (4.3.3) follows.
This proves the first assertion of the proposition.  The second assertion 
follows from (4.1.3).   
\end{proof}

\para{4.4.}
We shall determine the polynomial $u^{\pm}_{\Bla,\Bla}(\Bt)$. 
In the proof of Proposition 4.3, 
the equality $\Bmu = \Bla$ holds if and only if  
$w(\Bla) = \Bla$ and the equality holds for each $k$ in the formulas (4.3.7), 
namely, $s_{\nu,\nu'} = 1$ for any $s_{\nu,\nu'}$ appearing in the expression of 
$A^{\pm}_k, B^{\pm}_k$. It follows that 
\begin{equation*}
u^{\pm}_{\Bla,\Bla}(\Bt) = \sum_{w \in S_{\Bla}}\ve(w)\prod_{\nu'} 
            (-t_{\nu'})^{s_{w\iv(\nu'),w\iv(\nu)}}.
\end{equation*} 
Let $S'_{\Bla}$ be the subgroup of $S^r_m$ as given in 4.1.   
In order to obtain the equality for $B_k^{\pm}$ in (4.3.7) for any $a$, 
we must have $w \in S_{\Bla}'$.
In that case, $s_{w\iv(\nu'),w\iv(\nu)} = 1$ only when the pair $(\nu, \nu')$ 
satisfies the condition that  
\begin{align*}
  \nu > \nu_0, \ \nu = (k,i), \ \nu' = (k, j), \ i < j, \ w\iv(i) > w\iv(j). 
\end{align*}
Since $t_{\nu'} = t_{b(\nu)}$ in this case, we have
\begin{equation*}
\tag{4.4.1}
u^{\pm}_{\Bla,\Bla}(\Bt) = \prod_{k=1}^r\sum_{w_k \in S_{\la'_k}}t_k^{\ell(w_k)} 
                         = v'_{\Bla}(\Bt),
\end{equation*}           
where $\ell(w_k)$ is the length function of $S_{\la'_k}$. 
\par 
By taking Proposition 4.3 and 
(4.4.1) into account, we have an expression 
\begin{equation*}
\tag{4.4.2}
\wt P^{\pm}_{\Bla}(x;\Bt) = s_{\Bla}(x) + 
         \sum_{\Bmu <  \Bla}w^{\pm}_{\Bla,\Bmu}(\Bt)s_{\Bmu}(x)
\end{equation*} 
with $w^{\pm}_{\Bla,\Bmu}(\Bt) \in \BZ[\Bt]$.

%%%
%%%
%%%
\par\medskip
\section{Closed formula for Hall-Littlewood functions -- continued}

\para{5.1.}
In this section, we discuss the expansion of $\wt Q^{\pm}_{\Bla}$ in terms of 
$q^{\pm}_{\Bmu}$. 
\par
For given $1 \le a \le r$ and $m \ge 1$, we define a subset $\SM^{(a)}$ of $\SM$ by 
removing $(1,1), \dots, (a-1,1)$ from $\SM$.   
Note that if $a = 1$, $\SM^{(1)}$ coincides with the original $\SM$.
The total order on $\SM^{(a)}$ is 
inherited from $\SM$. 
We consider  an $r$-partition $\Bla = (\la^{(1)}, \dots, \la^{(r)})$, 
where $\la^{(k)} = (\la^{(k)}_2, \la^{(k)}_3, \dots, \la^{(k)}_m)$ for $k < a$ 
and $\la^{(k)} = (\la^{(k)}_1, \dots, \la^{(k)}_m)$ for $k \ge a$.  Thus we have 
a bijective correspondence $\la^{(k)}_i \lra (k,i) \in \SM^{(a)}$.
In this case, we say that $\Bla$ is compatible with $\SM^{(a)}$. 
For such $\Bla$, one can construct a polynomial 
$R^{\pm}_{\Bla}(x;\Bt)$ by extending the previous definition, 
where $x = \{ x^{(k)}_i \mid (k,i) \in \SM^{(a)} \}$.  
(Here we consider $S_{\Bm} = S_{m -1}^{a-1} \times S_m^{r-a+1}$ as 
the permutation group of the variables 
$\{ x_{\nu} \mid \nu \in \SM^{(a)}\}$ for $\Bm = (m-1, \dots, m-1, m, \dots, m)$.)
\par
In the following discussion, by fixing $a$, we write $\SM^{(a)}$ as $\SM$. 
We discuss separately  the ``$+$''case and the ``$-$''case.  First consider the
``$-$''case. Then $\wt Q^-_{\Bla}$ is defined by using the formula (4.2.5).
Let $b \le r$ be the smallest integer such that 
$ b \ge a$ and that $\la^{(b)} \ne \emptyset$, 
hence $\la^{(b)}_1 \ne 0$. If such $b$ does not exist, i.e., $\la^{(a')} = \emptyset$ 
for any $r \ge a' \ge a$, we put $b = r$. 
\par
Let $\Bmu = (\mu^{(1)}, \dots, \mu^{(r)})$ be the $r$-partition compatible with 
$\SM^{(b+1)}$ 
obtained from $\Bla$ by removing $\{ \la^{(a')}_1 \mid a \le a' \le b\}$.  Thus we have 
$\mu^{(k)}= (\la^{(k)}_2, \la^{(k)}_3, \dots, \la^{(k)}_m)$ for $1 \le k \le b$,
and $\Bmu^{(k)} = \Bla^{(k)}$ otherwise. 
Put $\SM' = \SM ^{(b+1)}$, and consider the polynomial $\wt Q^-_{\Bmu}(x';\Bt)$ for 
$x' = (x_{\nu})_{\nu \in \SM'}$.
(Note that if $b = r$, $\SM'$ coincides with $\SM^{(1)}$, but by replacing $m$ by $m-1$.)  
\par
$\Bm'$ is defined for $\SM '$ similarly to $\Bm$ and we have 
$S_{\Bm'} =  S_{m-1}^{b} \times S_m^{r-b}$.  Hence 
$S_{\Bm}/S_{\Bm'} \simeq [1,m]^{b-a+1}$, and we express the elements in 
$S_{\Bm}/S_{\Bm'}$ as $\Bi = (i_a, i_{a+1}, \dots, i_b) \in [1,m]^{b-a+1}$.
For 
$\Bi \in S_{\Bm}/S_{\Bm'}$, 
we denote by 
$\wt Q^{[\Bi],-}_{\Bmu}$ the polynomial obtained from $\wt Q^-_{\Bmu}$ 
by replacing the variables 
$x^{(k)}_2, \dots, x^{(k)}_m$ by 
$x^{(k)}_1, \dots, \wh x^{(k)}_{i_k}, \dots, x^{(k)}_m$ ($\wh x^{(k)}_{i_k}$ is removed) 
for $a \le k \le b$,
and leaving variables in other rows unchanged.  
We have the following lemma.
%%%%
%%%%
\begin{lem} %%%% Lemma 5.2
Under the notation above, we have
\begin{equation*}
\wt Q^-_{\Bla} = \sum_{\Bi \in S_{\Bm}/S_{\Bm'}}(x_{i_b}^{(b)})^{\la^{(b)}_1 - \ve^{(b)}_{1,-}}
              g^{(a)}_{\Bi,-}\wt Q^{[\Bi],-}_{\Bmu}, 
\end{equation*}
where $g^{(a)}_{\Bi,-}$ for $\Bi = (i_a, \dots, i_b) \in [1,m]^{b-a+1}$ 
 is given by 
\begin{equation*}
\tag{5.2.1}
g^{(a)}_{\Bi,-}(x;\Bt) = \frac{\displaystyle\biggl(\prod_{\substack{a \le k < b }}
      \prod_{\substack{j \ge 1 \\  j \ne i_{k+1}}}
(x^{(k)}_{i_k} - t_kx^{(k+1)}_j)\biggr) 
    \prod_{\substack{j \ge 1}}
    (x_{i_b}^{(b)} - t_bx^{(b + 1)}_j)}
     {\displaystyle\prod_{\substack{a \le k \le b }}
      \prod_{\substack{ j \ne i_k}}(x^{(k)}_{i_k} - x^{(k)}_j)} 
\end{equation*}
if $b < r$, 
\begin{equation*}
\tag{5.2.2}
g^{(a)}_{\Bi,-}(x;\Bt) = \frac{\displaystyle\biggl(\prod_{\substack{a \le k < r}}
      \prod_{\substack{j \ge 1 \\  j \ne i_{k+1}}}
(x^{(k)}_{i_k} - t_kx^{(k+1)}_j)\biggr) 
    \prod_{\substack{ j \ge 2}}
    (x_{i_r}^{(r)} - t_rx^{(1)}_j)} 
     {\displaystyle\prod_{\substack{a \le k \le r }}
    \prod_{\substack{ j \ne i_k}}(x^{(k)}_{i_k} - x^{(k)}_j)} 
\end{equation*}
if $b = r$ and $a > 1$,  
\begin{equation*}
\tag{5.2.3}
g^{(1)}_{\Bi,-}(x;\Bt) = \frac{\displaystyle (1 - t_0) 
          \biggl(\prod_{\substack{1 \le k < r}}
      \prod_{\substack{j \ge 1 \\  j \ne i_{k+1}}}
(x^{(k)}_{i_k} - t_kx^{(k+1)}_j)\biggr) 
    \prod_{\substack{ j \ge 2}}
    (x_{i_r}^{(r)} - t_rx^{(1)}_j)} 
     {\displaystyle\prod_{\substack{1 \le k \le r }}
    \prod_{\substack{ j \ne i_k}}(x^{(k)}_{i_k} - x^{(k)}_j)} 
\end{equation*}
if $b = r$ and $a = 1$.
\end{lem}

\begin{proof}
Here $S_{\Bm'}$ is the stabilizer of the variables $x^{(k)}_1$ for $a \le k \le b$ 
in $S_{\Bm}$. Note that $S'_{\Bla} = S'_{\Bmu} \subset S_{\Bm'}$. First assume that $b < r$.  
Let $\Bi_0 = (1, \dots, 1) \in [1,m]^{b-a+1}$.  
Since $(x^{(b)}_1)^{\la^{(b)}_1 - \ve^{(b)}_{1,-}}g^{(a)}_{\Bi_0,-}$ is stable by $S_{\Bm'}$ 
we have 
\begin{align*}
(1 - t_0)^{j_0}\sum_{w \in S_{\Bm'}/S'_{\Bmu}}&w\biggl\{
     \prod_{1 \le k \le r}(x^{(k)})^{\la^{(k)} - \ve^{(k)}}
            \prod_{\nu \le \nu_0} I^-_{\nu}(x;\Bt)\big/
          \prod_{(k,i) \le \nu_0}\prod_{i < j}(x^{(k)}_i - x^{(k)}_j)
             \biggr\}  \\
      &= (x^{(b)}_1)^{\la^{(b)}_1 - \ve^{(b)}_{1,-}}g^{(a)}_{\Bi_0,-}\wt Q^-_{\Bmu}(x';\Bt).
\end{align*}
(Here $j_0$ used in the definition of $\wt Q^-_{\Bla}$ 
coincides with $j_0$ used for $\wt Q^-_{\Bmu}$.) 
Hence
\begin{equation*}
\tag{5.2.4}
\wt Q^-_{\Bla}(x;\Bt) = \sum_{w \in S_{\Bm}/S_{\Bm'}}w\left\{ 
      (x^{b)}_1)^{\la^{(b)}_1 - \ve^{(b)}_{1,-}}g^{(a)}_{\Bi_0}\wt Q^-_{\Bmu}(x';\Bt)\right\}. 
\end{equation*}
The lemma follows from this. 
Next assume that $b = r$. 
Since $\prod_{j \ge 2}(x^{(r)}_1 - t_rx^{(1)}_j)$ is stable by the action of 
$S_{m-1} \times S_{m-1}$, 
$(x^{(r)}_1)^{\la^{(r)}_1 - \ve^{(r)}_{1,-}}g^{(a)}_{\Bi_0,-}$ is again stable by 
$S_{\Bm'}$. 
Then a similar argument works.
Note that in the case of (5.2.2), $j_0$ is common for $\wt Q^-_{\Bla}$ and for
$\wt Q^-_{\Bmu}$, but in the case of (5.2.3), 
the discrepancy for $j_0$ between $\wt Q^-_{\Bla}$ 
and $\wt Q^-_{\Bmu}$ occurs.   
The lemma is proved.  
\end{proof}

\remark{5.3.} \
A similar formula was proved in Lemma 3.3 in [S1]. 
But in our definition of $R^{\pm}_{\Bla}(x;\Bt)$, we can not
take $b = a$ if $\la^{(a)}_1 = 0$ for $a < r$ since the product 
$\prod_{j \ge 2}(x_1^{(a)} - t_ax_j^{(a+1)})$ is not stable 
by the action of $S_{m-1} \times S_m$. 

\para{5.4.}
We consider the ``$+$''case. Let $\SM = \SM^{(a)}$ be as in 5.1.  
In  this case, we assume that $\la_2^{(1)} \ne 0$.
Then $\wt Q^+_{\Bla}$ is defined as in (4.2.5). 
Put $\SM' = \SM^{(a+1)}$, and consider $\Bm'$ with respect to $\SM '$.
We have $S_{\Bm'} = S_{m-1}^a \times S_m^{r-a}$, and $S_{\Bm}/S_{\Bm'} \simeq [1,m]$.  
Put, for $i = 1, \dots, m$,
\begin{align*}
g^{(a)}_{i,+}(\Bt) = \prod_{\substack{ (a-1, j) \in \SM }}
                   (x_i^{(a)} - t_{a-1}x_j^{(a-1)})
             \big / \prod_{\substack{ j \ne i}}(x^{(a)}_i - x^{(a)}_j).
\end{align*}
Here the condition $(a-1, j) \in \SM$ is given by $j \ge 2$ if $a > 1$ and
$j \ge 1$ if $a = 1$.  Then the product $\prod_{j \ge 2}(x^{(a)}_1 - t_{a-1}x^{(a-1)}_j)$ 
(resp. $\prod_{j \ge 1}(x^{(1)}_1 - t_{r}x^{(r)}_j)$) for $a > 1$ (resp. $a = 1$)
is stable by the action of $S_{\Bm'}$. (Note that the condition $j \ge 1$ in the $a = 1$ case
comes from the assumption that $\la^{(1)} \ne \emptyset$.  If $\la^{(1)} = \emptyset$, 
we need the condition $j \ge 2$, in which case, the product is not stable by 
$S_{\Bm'}$.)   
\par
Now $\Bmu$ is defined as in 5.1, and $\wt Q^{[i],+}_{\Bmu}$ can be defined for 
$i = 1, \dots, m$ (apply 5.1 to the case where $b = a$).  
The following formula can be proved in a similar way as in Lemma 5.2
(Note, by the condition $\la^{(1)}_2 \ne 0$, that $j_0$ is common for $\Bla$ and $\Bmu$.)

\begin{lem}  %%%%  Lemma 5.5
Assume that $\la^{(1)}_2 \ne 0$. 
Then we have
\begin{equation*}
\wt Q^+_{\Bla} = 
  \sum_{i = 1}^m(x_i^{(a)})^{\la_1^{(a)}-\ve^{(a)}_{1,+}}g_{i,+}^{(a)}\wt Q^{[i],+}_{\Bmu}.
\end{equation*}
\end{lem}

The following lemmas will be used in later discussions.
%%%%
%%%%
\begin{lem}  %%%%  Lemma 5.6
Assume that $a < b$. Consider 
the variables $x^{(k)}$ for $a \le k \le b$, and denote the
action of $S_m$ on the variable $x^{(k)}$ by $S_m^{(k)}$. 
Then we have
\begin{equation*}
\tag{5.6.1}
\sum_{w \in S^{(a)}_m \times \cdots \times S_m^{(b-1)}}
          w\biggl(\prod_{a \le k < b}\prod_{2 \le j \le m}
         \frac{x^{(k)}_1 - t_kx^{(k+1)}_j}{x_1^{(k)} - x^{(k)}_j} \biggr) 
            = |S_{m-1}|^{b-a}.
\end{equation*}
\end{lem}

\begin{proof}
If we write $S_m^{(a)} \times \cdots \times S_m^{(b-1)}$ as $S_m^{b-a}$,  
the left hand side of (5.6.1) is equal to
\begin{align*}
        \sum_{w \in S_m^{b-a}}
           &w\biggl(\frac{\prod_{a \le k < b}
           \prod_{j \ge 2}(x^{(k)}_1 - t_kx^{(k+1)}_j)\prod_{2 \le i < j}(x^{(k)}_i - x^{(k)}_j)}
              {\prod_{a \le k < b}\prod_{1 \le i < j \le m}(x^{(k)}_i - x^{(k)}_j)}\biggr) \\  \\
          &= \frac {\sum_{w \in S_m^{b-a}}\ve(w)w\biggl(
          \prod_{a \le k < b}\prod_{j \ge 2}(x^{(k)}_1 - t_kx^{(k+1)}_j)
              \prod_{2 \le i < j}(x^{(k)}_i - x^{(k)}_j)\biggr)}
                  {\prod_{a \le k < b}\prod_{i < j}(x^{(k)}_i - x^{(k)}_j)}. 
\end{align*}
But the sum in the numerator is an alternating polynomial with respect to the 
variables $x^{(k)}$ for $a \le k < b$, 
and so is divisible by $\prod_{a \le k < b}\prod_{i < j}(x^{(k)}_i - x^{(k)}_j)$.  
Hence by comparing the degrees as polynomials with respect to the variables 
$x_1^{(k)}$ for a fixed $k$, 
the last formula is equal to 
\begin{align*}
        \frac{1}{\prod_{a \le k < b}\prod_{i < j}(x^{(k)}_i - x^{(k)}_j)}
                \sum_{w \in S_m^{b-a}}\ve(w)w\biggl(\prod_{a \le k < b}(x^{(k)}_1)^{m-1}
                   \prod_{2 \le i < j}(x^{(k)}_i - x^{(k)}_j)\biggr),
\end{align*}
where the numerator coincides with 
\begin{equation*}
\prod_{a \le k < b}\sum_{w \in S_m^{(k)}}\ve(w)w\biggl(
                   (x^{(k)}_1)^{m-1}
                   \prod_{2 \le i < j}(x^{(k)}_i - x^{(k)}_j)\biggr).
\end{equation*}
If we put $x^{(k)}_i = y_i, S_m^{(k)} = S_m$,  we have 
\begin{align*}
\sum_{w \in S_m}\ve(w)w&\biggl(y_1^{m-1}\prod_{2 \le i < j \le m}(y_i - y_j)\biggr) \\
    &= \sum_{w \in S_m}\ve(w)w\sum_{w' \in S_{m-1}}\ve(w')w'(y_1^{m-1}y_2^{m-2}\cdots y_{m-1}) \\
    &= |S_{m-1}|\sum_{w \in S_m}\ve(w)w(y_1^{m-1}y_2^{m-2}\cdots y_{m-1}) \\
    &= |S_{m-1}|\prod_{i< j}(y_i - y_j).
\end{align*}
The lemma is proved.  
\end{proof}

\begin{lem}   %%%%   Lemma 5.7
Consider three types of variables $x_i, y_i, z_i$ for 
$i = 1, \dots, m$. 
Then the following identity holds.
\begin{equation*}
\tag{5.7.1}
\sum_{w \in S_m}w\biggl(\frac{\prod_{j \ge 2}(x_1 - t_1y_j)\prod_{j \ge 2}(z_1 - t_3x_j)}
                        {\prod_{j \ge 2}(x_1 - x_j)}\biggr)
        = |S_{m -1}|\prod_{j \ge 2}(z_1 - t_1t_3y_j), 
\end{equation*} 
where $S_m$ acts on the variables $x_1, \dots, x_m$ as permutations. 
\end{lem} 

\begin{proof}
We define an operator $\SA_x$ on the variables $x_1, \dots, x_m$ by 
\begin{equation*}
\SA_x = \prod_{i < j}(x_i - x_j)\iv \sum_{w \in S_m}\ve(w)w.
\end{equation*} 
Then the left hand side of (5.7.1) can be written as 
\begin{equation*}
\SA_x\biggl(\prod_{j \ge 2}(x_1 - t_1y_j)\prod_{j \ge 2}(z_1 - t_3x_j)
            \prod_{2 \le i < j}(x_i - x_j)\biggr).
\end{equation*}
We can write
\begin{align*}
\prod_{j \ge 2}(x_1 - t_1y_j) &= \sum_{k = 0}^{m -1}(-t_1)^kx_1^{m - k-1}\sum_{\substack{
                   I \subset [2, m] \\ |I| = k}}y_I, \\
\prod_{j \ge 2}(z_1 - t_3x_j) &= \sum_{\ell = 0}^{m -1}(-t_3)^{\ell}z_1^{m -\ell -1}
                     \sum_{\substack{  J \subset [2, m] \\  |J| = \ell}}x_J,  \\
\prod_{j \ge 2}(z_1 - t_1t_3y_j) &= \sum_{k = 0}^{m -1}(-t_1t_3)^kz_1^{m -k-1}
                      \sum_{\substack{J \subset [2, m] \\ |J| = k}}y_J,
\end{align*}
where $y_I = y_{i_1}\cdots y_{i_k}$ for $I = \{ i_1, \dots, i_k\}$, and similarly 
for $x_J$. 

Put $F = \prod_{j \ge 2}(x_1 - t_1y_j)\prod_{j \ge 2}(z_1 - t_3x_j)$.  Then 
\begin{align*}
F = \sum_{k = 0}^{m -1}\sum_{\ell = 0}^{m -1}
         \sum_{\substack{  I \subset [2, m] \\  |I| = k}}
         \sum_{\substack{  J \subset [2, m] \\  |J| = \ell}}
       (-t_1)^kx_1^{m -1-k}y_Ix_J(-t_3)^{\ell}z_1^{m -1-\ell}.
\end{align*}
We compute, for a fixed $k, \ell$, 
\begin{align*}
\sum_{\substack{ J \subset [2,m] \\ |J| = \ell}}
      &x_Jx_1^{m -1-k}\prod_{2 \le i < j}(x_i - x_j) \\
    &= \sum_{\substack{ J \subset [2,m ] \\ |J| = \ell}}x_Jx_1^{m -1-k}
          \sum_{\s \in S_{m -1}}\ve(\s)x_{\s(2)}^{m -2}x_{\s(3)}^{m -3}\cdots x_{\s(m -1)},
\end{align*} 
where $S_{m-1}$ is the stabilizer of $x_1$ in $S_m$. 
We apply the operator $\SA_x $ 
for each monomial $X = x_Jx_1^{m -1-k}x^{m -2}_{\s(2)}\cdots x_{\s(m -1)}$ 
determined by the choice of $J$ and $\s$. 
If $k > \ell$, then $\SA_x(X) = 0$ by the degree reason. 
So assume that $k \le \ell$.   
We write 
\begin{equation*}
x_1^{m -1-k}x_{\s(2)}^{m -2}x_{\s(3)}^{m -3}\cdots x_{\s(m -1)}
   = x_{\s(2)}^{m -2}\cdots x_{\s(k+1)}^{m - k - 1}x_1^{m - k -1}
          x_{\s(k + 2)}^{m -k-2}\cdots x_{\s(m -1)}.
\end{equation*}
Note that if $a_1, \dots, a_m$ are not all distinct in the monomial
$X = x_1^{a_1}\cdots x_m^{a_m}$, then again $\SA_x(X) = 0$. 
It follows that, if $\SA_x(X) \ne 0$, then $X$ must have the form
\begin{equation*}
X = x_{\s(2)}^{m -1}\cdots x_{\s(k+1)}^{m - k}x_1^{m - k -1}
          x^{m -k-2}_{\s(k+2)}\cdots x_{\s(m -2)}, 
\end{equation*}
namely, $\ell = k$ and $J = \{ \s(2), \dots, \s(k+1)\}$.
Write $X$ as $x^{m -1}_{\t(1)}x^{m -2}_{\t(2)}\cdots x_{\t(m -1)}$
for $\t \in S_m$.  Then we have
\begin{equation*}
   \SA_x(X) = \ve(\tau) = (-1)^{k}\ve(\s).
\end{equation*} 
For a fixed $J$, the number of such $\s$ is equal to $k! \times (m -1-k)!$. 
The number of the choices of $J$ is $\binom{m -1}{k}$.  Hence
\begin{equation*}
   \SA_x\biggl(
     \sum_{\substack{ J \subset [2,m] \\ |J| = k}}
              (t_1t_3)^{k}x_Jx_1^{m -1-k}\prod_{2 \le i < j}(x_i - x_j)\biggr)
      = (m -1)!(-t_1t_3)^{k}.
\end{equation*}
It follows that 
\begin{align*}
|S_{m-1}|\iv\SA_x\biggl(F \prod_{2 \le i < j}(x_i - x_j)\biggr)
     &= \sum_{k = 0}^{m -1}\sum_{\substack{  J \subset [2, m] \\  |J| = k}}
           (-t_1t_3)^kz_1^{m -1-k}y_J  \\
     &= \prod_{j \ge 2}(z_1 - t_1t_3y_j).
\end{align*}
This proves (5.7.1).  The lemma is proved.
\end{proof}

\para{5.8.}
Recall the definition of $q^{(k)}_{s,\pm}(x;t)$ in (2.5.1).  
In the ``$-$''case with $k = r$, we define a function $\wt q^{(r)}_{s,-}(x;t)$ by 
\begin{equation*}
\tag{5.8.1}
\wt q^{(r)}_{s,-}(x;t) = \sum_{1 \le i \le m}(x_i^{(r)})^s
   \frac{\prod_{j \ge 2}(x_i^{(r)} - tx_j^{(1)})}{\prod_{j \ne i}(x_i^{(r)} - x_j^{(r)})}
\end{equation*}
for $s \ge 0$.  If $s \ge 1$, we have 
$q^{(r)}_{s,-}(x;t)|_{x^{(1)}_1 = 0} = \wt q^{(r)}_{s,-}(x;t)$.  
We note that $\wt q^{(r)}_{s,-} = 1$ if $s = 0$.  In fact,    
\begin{align*}
\sum_{1 \le i \le m}
   \frac{\prod_{j \ge 2}(x^{(r)}_i - tx^{(1)}_j)}{\prod_{j \ne i}(x^{(r)}_i - x^{(r)}_j)}  
      = |S_{m-1}|\iv \sum_{w \in S^{(r)}_m}w\biggl(
        \frac{\prod_{j \ge 2}(x^{(r)}_1 - tx^{(1)}_j)}
          {\prod_{j \ge 2}(x^{(r)}_1 - x^{(r)}_j)}\biggr) 
      = 1.
\end{align*}
The last equality follows from Lemma 5.6 by applying it to the case 
where $b-a = 1$.  
Thus $q^{(r)}_{s,-}(x;t)|_{x^{(1)}_1 = 0} = \wt q^{(r)}_{s,-}(x;t)$, and 
$\wt q^{(r)}_{s,-} \in \BZ[x;t]$ for $s \ge 0$. 
By using the generating function of $q^{(r)}_{s,-}(x; t_r)$ in (2.5.2), we have 
\begin{equation*}
\tag{5.8.2}
\sum_{s \ge 0}\wt q^{(r)}_{s,-}(x;t)u^s = \prod_{i = 2}^m (1 - tux_i^{(1)})
             \bigg /\prod_{i=1}^m (1 - ux^{(r)}_i). 
\end{equation*}
\par
In the ``$+$''case with $k > 1$, we define $\wt q^{(k)}_{s,+}(x;t)$ by 
\begin{equation*}
\tag{5.8.3}
\wt q^{(k)}_{s,+}(x;t) = \sum_{1 \le i \le m}(x_i^{(k)})^s
   \frac{\prod_{j \ge 2}(x_i^{(k)} - tx_j^{(k-1)})}{\prod_{j \ne i}(x_i^{(k)} - x_j^{(k)})}
\end{equation*}
for $s \ge 0$.
As in the ``$-$''case, we see that 
$q^{(k)}_{s,+}(x;t)|_{x^{(k-1)}_1 = 0} = \wt q^{(k)}_{s,+}(x;t)$, and that 
the generating function of $\wt q^{(k)}_{s,+}$ is given by 
\begin{equation*}
\tag{5.8.4}
\sum_{s \ge 0}\wt q^{(k)}_{s,+}(x;t)u^s = \prod_{i = 2}^m (1 - tux_i^{(k-1)})
             \bigg /\prod_{i=1}^m (1 - ux^{(k)}_i). 
\end{equation*}

In the ``$-$''case, we need the following.

\begin{prop}   %%%%   Prop 5.9
Assume that $a \le b \le r$, and put $d = b-a +1$.  Set 
\begin{equation*}
G^{a,b}_s(x;\Bt) = \sum_{\Bi \in [1,m]^d}(x^{(b)}_{i_b})^{s-\ve}g^{(a)}_{\Bi,-}(x;\Bt),
\end{equation*}
where $\ve = 1$ if $b < r$ and $\ve = 0$ if $b = r$. 
Furthermore assume that $s \ge 1$ if $b < r$. 
Then we have
\begin{equation*}
G^{a,b}_s(x;\Bt) = \begin{cases}
                    q^{(b)}_{s,-}(x;t_b) &\quad\text{ if } b < r, \\ 
                    \wt q^{(r)}_{s,-}(x;t_r)  &\quad\text{ if $b = r$ and $a > 1$},  \\
                    q_s(x^{(r)}; t_0) &\quad\text{ if $b = r$ and $a = 1$},
                   \end{cases}
\end{equation*}
where $t_0 = t_1\cdots t_r$, and  $q_s(x^{(r)};t_0)$ is 
the original $q$-function given in (2.3.4) with respect to the variables $x^{(r)}$. 
\end{prop}

\begin{proof}
First assume that $b < r$. 
We denote by $S_m^{(k)}$ the action of $S_m$ on the variables $x^{(k)}$. 
\begin{align*}
&G^{a,b}_s(x;\Bt)  \\ 
&= \sum_{\Bi \in [1,m]^d}(x^{(b)}_{i_b})^{s-1}
    \frac{\prod_{a \le k < b}\prod_{\substack{j \ge 1 \\  j \ne i_{k+1}}}
(x^{(k)}_{i_k} - t_kx^{(k+1)}_j)\cdot 
    \prod_{j \ge 1}
    (x_{i_b}^{(b)} - t_bx^{(b + 1)}_j)}
     {\prod_{a \le k \le b}\prod_{j \ne i_k}(x^{(k)}_{i_k} - x^{(k)}_j)}  \\ \\
    &= |S_{m-1}|^{-d}\!\!\!\!\!\sum_{w \in S_m^{(a)} \times \cdots \times S_m^{(b)}}w\biggl(
    \frac{\prod_{a \le k < b}\prod_{\substack{j \ge 2}}
(x^{(k)}_1 - t_kx^{(k+1)}_j) \cdot  
    (x^{(b)}_1)^{s-1}\prod_{j \ge 1}
    (x_1^{(b)} - t_bx^{(b + 1)}_j)}
     {\prod_{a \le k \le b}\prod_{j \ge 2}(x^{(k)}_1 - x^{(k)}_j)}\biggr)  \\  \\
   &= |S_{m-1}|^{-d}\sum_{w' \in S_m^{(b)}}w'\biggl(   
            A \frac{(x_1^{(b)})^{s-1}\prod_{j \ge 1}(x_1^{(b)} - t_bx_j^{(b+1)})}
                 {\prod_{j \ge 2}(x_1^{(b)} - x_j^{(b)})}\biggr),   
\end{align*}
where
\begin{equation*}
A = \sum_{w \in S_m^{(a)} \times \cdots \times S_m^{(b-1)}}
       w \biggl(\frac{\prod_{a \le k < b}\prod_{j \ge 2}(x_1^{(k)} - t_kx_j^{(k+1)})}
                   {\prod_{a \le k < b}\prod_{j \ge 2}(x_1^{(k)} - x_j^{(k)})}\biggr). 
\end{equation*}
By Lemma 5.6, we have $A = |S_{m -1}|^{b-a}$.
Since 
\begin{equation*}
\sum_{w' \in S_m^{(b)}}w'\biggl(   
            \frac{(x_1^{(b)})^{s-1}\prod_{j \ge 1}(x_1^{(b)} - t_bx_j^{(b+1)})}
                 {\prod_{j \ge 2}(x_1^{(b)} - x_j^{(b)})}\biggr)
        = |S_{m-1}|q^{(b)}_{s,-}(x;t_b),
\end{equation*}
we obtain the required formula.
A similar argument works also for the case where $b = r$ and $a > 1$.  
In that case, the formula
in the last step is given by
\begin{equation*}
\sum_{w' \in S_m^{(r)}}w'\biggl(   
            \frac{(x_1^{(r)})^{s}\prod_{j \ge 2}(x_1^{(r)} - t_rx_j^{(1)})}
                 {\prod_{j \ge 2}(x_1^{(r)} - x_j^{(r)})}\biggr) 
      = |S_{m-1}|\wt q^{(r)}_{s,-}(x;t_r).
\end{equation*} 
Thus the assertion holds.  
\par
Finally consider the case where $b = r$ and $a = 1$.
In this case, we have
\begin{align*}
&(1 - t_0)\iv G^{1,r}_s(x;\Bt)  \\ 
&= |S_{m-1}|^{-r} \sum_{w \in S_m^r}w\biggl((x_1^{(r)})^s
        \prod_{1 \le k < r}\prod_{j \ge 2}
          \frac{x_1^{(k)} - t_kx_j^{(k+1)}}{x_1^{(k)} - x_j^{(k)}}
    \prod_{j \ge 2}\frac{x_1^{(r)} - t_rx_j^{(1)}}{x_1^{(r)} - x_j^{(r)}}\biggr).
\end{align*}
By applying Lemma 5.7 for $x_i = x_i^{(1)}, y_i = x_i^{(2)}, z_i = x_i^{(r)}$, 
the right hand side turns out to be
\begin{align*}
|S_{m -1}|^{-r+1} \sum_{w \in S_m^{r-1}}w\biggl((x_1^{(r)})^s
        \prod_{2 \le k < r}\prod_{j \ge 2}
          \frac{x_1^{(k)} - t_kx_j^{(k+1)}}{x_1^{(k)} - x_j^{(k)}}
    \prod_{j \ge 2}\frac{x_1^{(r)} - t_1t_rx_j^{(2)}}{x_1^{(r)} - x_j^{(r)}}\biggr).
\end{align*}
Thus by repeating this procedure, we see that
\begin{align*}
(1 - t_0)\iv G^{1,r}_s(x;\Bt) 
       &= |S_{m -1}|\iv\sum_{w \in S_m}w\biggl((x^{(r)}_1)^s\prod_{j \ge 2}
          \frac{x_1^{(r)} - t_1\cdots t_rx_j^{(r)}}{x_1^{(r)} - x_j^{(r)}}\biggr)  \\
                 &= (1 -t_0)\iv q_s(x^{(r)};t_0).
\end{align*}
Hence the assertion holds.  The proposition is proved.
\end{proof}

\remark{5.10.} Let $\la = (\la_1, \dots, \la_m)$ be a partition, and 
$Q_{\la}(y;t)$ the original Hall-Littlewood function. 
In this case it is known by [M, III, 2.14] that
\begin{equation*}
\tag{5.10.1}
Q_{\la}(y;t) = \sum_{i = 1}^m y_i^{\la_i}g_iQ^{[i]}_{\mu}(y;t)
\end{equation*}
with 
\begin{equation*}
g_i = (1-t)\prod_{j \ne i}\frac{y_i - ty_j}{y_i - y_j},
\end{equation*}
where $\mu = (\mu_2, \dots, \mu_m)$ and $Q^{[i]}_{\mu}$ is defined similarly 
to 5.1. 
Now assume that $\Bla = (-, \dots, -,\la^{(r)}) \in \SP_{n,r}$. 
We compare the formula (5.10.1) with Lemma 5.2.  Then by using 
a similar argument as in the proof of the third formula in Proposition 5.9,
one can show, by induction on $m$, that  $\wt Q^-_{\Bla}(x;\Bt)$ coincides with 
$Q_{\la^{(r)}}(x^{(r)}; t_0)$ for $t_0 = t_1\cdots t_r$.    
By Theorem 3.5 we know that $Q^-_{\Bla}(x;\Bt) = Q_{\la^{(r)}}(x^{(r)};t_0)$. 
It follows, for a special case $\Bla = (-,\dots, -,\la^{(r)})$, 
that we obtain
\begin{equation*}
\tag{5.10.2}
\wt Q_{\Bla}^-(x;\Bt) = Q_{\Bla}^-(x;\Bt).
\end{equation*}   

\para{5.11.}
Based on the discussion in 5.8, we define functions 
$\Psi^{(k,i)}_{-}(u), \Psi^{(k)}_+(u)$ associated to $\SM$ as follows. 
Set $\vD = \vD(\Bla) = \{ (k,i) \in \SM  \mid \la^{(k)}_i  \ne 0 \}$. Let  
$\vD_0$ be the subset of $\vD$ consisting of $(r,i)$ such that 
$\la^{(k)}_i = 0$ for $1 \le k < r$, and set $\vD_1 = \vD - \vD_0$. 
\par
In the ``$-$'' case, for $(k,i) \in \vD_1$, set
\par\medskip
\begin{equation*}
\tag{5.11.1}
\Psi^{(k,i)}_{-}(u) = \prod_{\substack{i \le j \\  (k + 1,j) \in \SM }}
         (1 - t_k ux_j^{(k+1)}) \bigg/
                          \prod_{\substack{i \le j \\ (k,i) \in \SM }}(1 - ux_j^{(k)}). 
\end{equation*}
Also set, for $(r,i) \in \vD_0$ and $t_0 = t_1\cdots t_r$,  
\begin{equation*}
\tag{5.11.2}
\Psi^{(r,i)}_{-}(u) = \prod_{\substack{i \le j \\ (r,i) \in \SM }}
                 (1 - t_0ux_j^{(r)}) \bigg/
                          \prod_{\substack{i \le j \\ (r,i) \in \SM }}
                 (1 - ux_j^{(r)}). 
\end{equation*}
In the ``$+$''case, for $1 \le k \le r$, set 
\begin{equation*}
\tag{5.11.3}
\Psi^{(k)}_{+}(u) = \prod_{(k - 1,i) \in \SM}(1 - t_{k-1} ux_i^{(k-1)}) \bigg/
                          \prod_{(k,i) \in \SM }(1 - ux_i^{(k)}). 
\end{equation*}
\newpage
Then we have
\begin{equation*}
\tag{5.11.4}
\Psi^{(k,1)}_{-}(u) =       \begin{cases}
                    \displaystyle\sum_{s = 0}^{\infty} q^{(k)}_{s,-}(x;t_k)u^s 
                      &\quad\text{ if $(k,1) \in \vD_1$ and $k < r$,} \\ \\
                    \displaystyle\sum_{s = 0}^{\infty} \wt q^{(r)}_{s,-}(x;t_r)u^s 
                      &\quad\text{ if $(k,1) \in \vD_1$,  $k = r$ and $a > 1$,} \\  \\
                    \displaystyle\sum_{s = 0}^{\infty} q_{s}(x^{(r)};t_0)u^s 
                      &\quad\text{ if $(k,1) \in \vD_0$, $k = r$,} 
                             \end{cases} 
\end{equation*}
\begin{equation*}
\tag{5.11.5}
\Psi^{(k)}_+(u) =  \begin{cases} 
                    \displaystyle\sum_{s = 0}^{\infty} q^{(1)}_{s,+}(x;t_r)u^s 
                        &\quad\text{ if $k = 1$ and $a = 1$,} \\  \\ 
                     \displaystyle\sum_{s = 0}^{\infty} \wt q^{(k)}_{s,+}(x;t_{k-1})u^s 
                        &\quad\text{ if } 1 < a \le k \le r. 
                    \end{cases}
\end{equation*}
\par
We introduce infinitely many variables $u_1^{(k)}, u_2^{(k)}, \dots$ 
for $1 \le k \le r$. 
We consider the set $\wt\SM = \wt\SM^{(a)}$ by letting $m \mapsto \infty$, 
and give the total order on $\wt\SM$ inherited from $\SM$.  
We define functions $\Phi_{\pm}(\Bu)$ 
with multi-variables $\Bu = \{ u^{(k)}_i \mid (k,i) \in \wt\SM\}$ by 
\begin{align*}
\tag{5.11.6}
\Phi_{-}(\Bu) = &\prod_{\substack{(k,i) \in \vD}}
                  \prod_{j \ge i}\Psi^{(k,i)}_{-}(u^{(k)}_j)   \\
      &\times \frac{\prod_{\substack{i < j, (k,i) \in \vD_1 }}
         (1 - (u_i^{(k)})\iv u_j^{(k)})}
       {\prod_{\substack{(k,i) < (k-1,\ell) \\ (k,i), (k-1,i) \in \vD_1}}
                 (1 - t_{k-1}(u_i^{(k)})\iv u^{(k - 1)}_{\ell})}
            \prod_{\substack{i < j \\ (r,i) \in \vD_0 }}
         \frac{1 - (u_i^{(r)})\iv u_j^{(r)}}{1 - t_0(u_i^{(r)})\iv u^{(r)}_j},
\end{align*}
\begin{equation*}
\tag{5.11.7}
\Phi_{+}(\Bu) = \prod_{\substack{(k,i) \in \wt\SM }}\Psi^{(k)}_{+}(u^{(k)}_i)
     \prod_{\substack{i < j \\  (k,i) < (k +1,\ell)}}
         \frac{1 - (u_i^{(k)})\iv u_j^{(k)}}{1 - t_{k}(u_i^{(k)})\iv u^{(k + 1)}_{\ell}}. 
\end{equation*}
\par
Recall that $\nu_0 = (k_0, i_0)$ (see 4.1). 
We consider the following condition on $\Bla \in \SP_{n,r}$;
\par\medskip\noindent 
(A) : $\la^{(1)}_{i_0} \ne 0$. 
\par\medskip
We shall prove the following result.

\begin{prop}  %%%%   Prop. 5.12.
In the ``$+$''case, assume that $\Bla$ satisfies the condition \rm{(A)}.   
In the ``$-$''case, give no assumption. 
Then $\wt Q^{\pm}_{\Bla}(x;\Bt)$ coincides with the coefficient of 
$\Bu^{\Bla} = \prod_{k,i}(u^{(k)}_i)^{\la^{(k)}_i}$ in the function $\Phi_{\pm}(\Bu)$. 
\end{prop}

\begin{proof}
First we consider the ``$-$''case. 
We assume that either $a > 1$ or $a = 1$ and $\Bla$ is not of the form 
$(-,\dots, -,\la^{(r)})$.   
Let $b \le r$ be the smallest integer $b \ge a$ such that $\la^{(b)} \ne \emptyset$. 
We follow the notation in 5.1. First assume that such  $b$ exists.  
Hence $(b,1) \in \vD_1$.
Let $\Bi_0 = (1, \dots, 1) \in [1,m]^d$ be as before, and put 
$\Bu' = \Bu - \{ u^{(k)}_1 \mid a \le k \le b\}$. 
The function $\Phi^{[\Bi_0]}_-(\Bu')$ is defined similarly to $\Phi_-(\Bu)$,
by replacing $\Bu$ and $x_{\SM}$ by $\Bu'$ and $x_{\SM '}$. 
For each 
$\Bi = (i_a, \dots, i_b) \in [1,m]^d$ with $d = b-a+1$, 
we denote by  $\Phi_-^{[\Bi]}(\Bu)$ the function obtained 
from $\Phi^{[\Bi_0]}_-(\Bu')$ by replacing the variables 
$x^{(k)}_2, \dots, x^{(k)}_m$ by $x^{(k)}_1, \dots, \wh x^{(k)}_{i_k}, \dots x^{(k)}_m$
($\wh x^{(k)}_{i_k}$ is removed) for $a \le k \le b$.
Then by (5.11.6), we have

\begin{align*}
\tag{5.12.1}
\Phi^{[\Bi]}_-(\Bu') = \Phi_-(\Bu')\frac{\prod_{j \ge 2}1 - x^{(b)}_{i_b}u_j^{(b)}}
           {\prod_{(*)}1 - t_{b-1}x^{(b)}_{i_{b}}u^{(b-1)}_{\ell}},  
\end{align*}  
where the condition (*) is that $(b-1,1) \in \vD_1$ and $(b-1, \ell) \in \SM '$ 
(this occurs only when $a = b$). 
Moreover, $\Phi_-(\Bu')$ is defined as the product of factors in $\Phi_-(\Bu)$ 
not containing $u_1^{(b)}$. 
Let $\Bmu$ be as in 5.1.  
By induction hypothesis, we may assume that 
$\wt Q^{-,[\Bi_0]}_{\Bmu}$ is obtained as the coefficient of $\Bu^{\Bmu}$ in the 
function $\Phi^{[\Bi_0]}_-(\Bu')$. 
Thus a similar result holds also for $\wt Q^{-,[\Bi]}_{\Bmu}$.  
Combining it with Lemma 5.2, we see that $\wt Q^-_{\Bla}(x;\Bt)$ is the 
coefficient of $\Bu^{\Bla}$ in  
\begin{equation*}
\sum_{s \ge 0}(u_1^{(b)})^s\sum_{\Bi \in [1,m]^d}(x^{(b)}_{i_b})^{s - \ve^{(b)}_{1,-}}
             g^{(a)}_{\Bi,-}\Phi_-^{[\Bi]}(\Bu'),
\end{equation*}
where $\ve^{(b)}_{1,-} = 1$ (resp. $\ve^{(b)}_{1,-} = 0$) if $b < r$ (resp. $b = r$). 
Now by (5.12.1) this expression is equal to 
\begin{align*}
\tag{5.12.2}
\Phi_-(\Bu')\sum_{s \ge 0}(u_1^{(b)})^s\sum_{\Bi \in [1,m]^d}
         (x^{(b)}_{i_b})^{s - \ve^{(b)}_{1,-}}
             g^{(a)}_{\Bi,-}
              \frac{\prod_{j \ge 2}1 - x^{(b)}_{i_b}u_j^{(b)}}
           {\prod_{(*)}1 - t_{b-1}x^{(b)}_{i_{b}}u^{(b-1)}_{\ell}}.
\end{align*}
We expand the products in (5.12.2) as a power series in $x^{(b)}_{i_b}$, 
\begin{equation*}
\frac{\prod_{j \ge 2}1 - x^{(b)}_{i_b}u_j^{(b)}}
           {\prod_{(*)}1 - t_{b-1}x^{(b)}_{i_{b}}u^{(b-1)}_{\ell}}
                  = \sum_{p \ge 0}f_p(\Bu';\Bt)(x^{(b)}_{i_b})^p,
\end{equation*}
where $f_p(\Bu';\Bt)$ is a polynomial in $\Bu',\Bt$. 
Write the expression (5.12.2) as $\Phi_-(\Bu')Z$.  
Now assume that $b < r$.
Substituting the above expansion into 
$Z$, we have
\begin{align*}
Z &= \sum_{s \ge 0}(u^{(b)}_1)^s\sum_{p \ge 0}f_p\sum_{\Bi \in [1,m]^d}
         (x^{(b)}_{i_b})^{s - 1 + p}g^{(a)}_{\Bi,-}   \\
  &= \sum_{s, p \ge 0}(u^{(b)}_1)^s q^{(b)}_{s + p, -}f_p \\
  &= \sum_{j \ge 0}(u^{(b)}_1)^j q^{(b)}_{j, -}\sum_{p = 0}^j f_p(u_1^{(b)})^{-p},
 \end{align*}
where the second identity follows from Proposition 5.9.
By using (5.11.4), 
the positive degree part of $Z$ with respect to $u^{(b)}_1$ coincides with that of 
\begin{equation*}
\Psi_{-}^{(b,1)}(u^{(b)}_1)\sum_{p \ge 0}f_p(u^{(b)}_1)^{-p} 
          = \Psi_{-}^{(b,1)}(u^{(b)}_1)
\frac{\prod_{j \ge 2}1 - (u^{(b)}_1)\iv u_j^{(b)}}
           {\prod_{(*)}1 - t_{b-1}(u^{(b)}_1)\iv u^{(b-1)}_{\ell}}.
\end{equation*}
Thus $\wt Q^{-}_{\Bla}$ is the coefficient of $\Bu^{\Bla}$ in 
\begin{equation*}
\Psi_{-}^{(b,1)}(u^{(b)}_1)
\frac{\prod_{j \ge 2}1 - (u^{(b)}_1)\iv u_j^{(b)}}
           {\prod_{(*)}1 - t_{b-1}(u^{(b)}_1)\iv u^{(b-1)}_{\ell}}
     \cdot \Phi_-(\Bu')
      = \Phi_{-}(\Bu)
\end{equation*}
as asserted. 
\par
If $b = r$, in the last expression of $Z$, $q^{(b)}_{j,-}$ should be replaced by 
$\wt q^{(b)}_{j,-}$ by Proposition 5.9. By using (5.11.4), 
we obtain the required formula in this case. 
If such  $b$ does not exist, i.e., $\la^{(b)} = \emptyset$, still  
a similar formula as (5.12.1) holds, but we must replace the numerator by 1. 
The remaining argument is the same as above.
\par
Next consider the case where $\Bla = (-,\dots, -,\la^{(r)})$ and $a = 1$, 
namely, $(r,1) \in \vD_0$.
In this case, the formula (5.12.1) is replaced by 
\begin{equation*}
\tag{5.12.3}
\Phi^{[\Bi]}_-(\Bu') = \Phi_-(\Bu')\prod_{j \ge 2}
       \frac{1 - x^{(r)}_1u_j^{(r)}}
           {1 - t_0x^{(r)}_{1}u^{(r)}_j} = \Phi_-(\Bu')Z.  
\end{equation*}
Then a similar computation as above shows that the positive degree part of
$Z$ with respect to $u_1^{(r)}$ coincides with that of 
\begin{equation*}
\Psi^{(r,1)}_{-}(u_1^{(r)})\prod_{j \ge 2}\frac{1 - (u_1^{(r)})\iv u_j^{(r)}}
                  {1 - t_0(u_1^{(r)})\iv u_j^{(r)}}.
\end{equation*}
Hence the assertion holds by a similar argument as above. 
\par
Finally we consider the ``$+$''-case. 
The formula corresponding to (5.12.1) is given as 
\begin{align*}
\tag{5.12.4}
\Phi^{[i]}_+(\Bu') = \Phi_+(\Bu')\frac{\prod_{j \ge 2}1 - x^{(a)}_{i}u_j^{(a)}}
           {\prod_{(**)}1 - t_{a+1}x^{(a)}_{i}u^{(a+1)}_{\ell}},
\end{align*}  
where the condition (**) is that  $(a+1, \ell) \in \SM '$, namely, 
$\ell \ge 1$ if $a < r$ and $\ell \ge 2$ if $a = r$. 
Then a similar argument works by using Lemma 5.5 instead of Lemma 5.2.
The proposition is proved.   
\end{proof}

\para{5.13.}
Returning to the original setting, we consider $\SM = \SM^{(1)}$. 
Let $\Bbe = (\b^{(1)}, \dots, \b^{(r)})$ be an $r$-composition such that
$\sum_i|\b^{(i)}| = n$.  
We write $\b^{(k)}_i = \b_{\nu}$ if $\nu = (k,i) \in \SM$, and 
identify $\Bbe = (\b^{(k)}_i)$ with an element $(\b_{\nu}) \in \BZ^{\SM}$.  For 
$\nu \ne \nu' \in \SM$, we define an  operator 
$R_{\nu,\nu'}: \BZ^{\SM} \to \BZ^{\SM}$ as follows; for $\Bbe = (\b_{\xi}) \in \BZ^{\SM}$,  
$R_{\nu,\nu'}(\Bbe) = (\b'_{\xi})$ is given by 
\begin{equation*}
\b'_{\nu} = \b_{\nu} + 1, \qquad \b'_{\nu'} = \b_{\nu'} - 1, 
\end{equation*} 
and $\b'_{\xi} = \b_{\xi}$ for $\xi \ne \nu, \nu'$.
A raising operator is defined as a product of various 
$R_{\nu, \nu'}$ for $\nu < \nu'$.  
\par 
For each $r$-composition $\Bbe = (\b^{(k)}_i)$, one can define $q^{\pm}_{\Bbe}(x;\Bt)$ by 
generalizing the definition (2.5.3).
We also extend the definition of $q^{\pm}_{\Bbe}$ to the case where $\Bbe \in \BZ^{\SM}$, 
by putting $q^{\pm}_{\Bbe} = 0$ if some $\b^{(k)}_i \in \BZ_{< 0}$. Then the action of 
the raising operator $R$ on the functions $q^{\pm}_{\Bbe}$ is defined by 
$R(q^{\pm}_{\Bbe}) = q^{\pm}_{\Bbe'}$ with $\Bbe' = R(\Bbe)$.  
As a corollary to Proposition 5.12 we have the following. 

\begin{cor}  %%%%  Cor. 5.14
Under the same assumption as in Proposition 5.12, 
for each $\Bla \in \SP_{n,r}$, $\wt Q^{\pm}_{\Bla}$ is expressed as 

\begin{equation*}
\tag{5.14.1}
\wt Q^-_{\Bla} = \biggl(
                   \prod_{\substack{ \nu = (k,i) \\\nu' = (k',j) \\ \nu < \nu'}}
               \frac{\prod_{\substack{k = k', \nu \in \vD_1}}
            (1 - R_{\nu,\nu'})}
       {\prod_{\substack{k' = k -1 \\ (k,i) \in \vD_1 \\ (k',i) \in \vD_1}}
            (1  - t_{k'}R_{\nu,\nu'})}
         \prod_{\substack{\nu < \nu' \\ b(\nu) = b(\nu') \\ \nu \in \vD_0}}
            \frac{1 - R_{\nu,\nu'}}{1  - t_0R_{\nu,\nu'}}\biggr)q_{\Bla}^-,                 
\end{equation*}
\begin{equation*}
\tag{5.14.2}
\wt Q^{+}_{\Bla} 
      = \biggl(\prod_{\substack{\nu < \nu'}}
       \frac{\prod_{b(\nu) = b(\nu')}
            1 - R_{\nu,\nu'}}
       {\prod_{b(\nu') = b(\nu) + 1 }
            1  - t_{b(\nu)}R_{\nu,\nu'}}\biggr)q^+_{\Bla}.
\end{equation*} 
In particular, $\wt Q^{\pm}_{\Bla} \in \BZ[x;\Bt]$, and is expressed as

\begin{equation*}
\tag{5.14.3}
\wt Q^{\pm}_{\Bla} (x;\Bt) = 
q^{\pm}_{\Bla}(x;\Bt) + \sum_{\substack{\Bmu \in \SP_{n,r} \\ \Bmu > \Bla}}
                  c_{\Bla,\Bmu}(\Bt) q^{\pm}_{\Bmu}(x;\Bt)
\end{equation*} 
with $c_{\Bla,\Bmu}(\Bt) \in \BZ[\Bt]$. 
\end{cor}

\begin{proof}
First consider the ``$-$'' case. 
By applying (5.11.4) for the case $\SM = \SM^{(1)}$, we have 
$\prod_{(k,i) \in \vD_1}\prod_{j \ge i} 
        \Psi^{(k,i)}_{-}(u^{(k)}_j) = \sum_{\Bbe}q^-_{\Bbe}\Bu^{\Bbe}$, where 
$\Bbe = (\b^{(1)}, \dots, \b^{(r)})$ runs over all $r$-compositions
such that $\b^{(k)}_i = 0$ if $(k,i) \notin \vD_1$. 
Moreover, 
$\prod_{(r,i) \in \vD_0}\prod_{j\ge i}\Psi^{(r,i)}_{-}(u^{(r)}_j) 
                = \sum_{\b} q_{\b}(u^{(r)})^{\b}$
where $\b$ runs over all the compositions in $\BZ_{\ge 0}^{j_0}$ ($j_0$ is as in 
(4.2.4)) and $q_{\b} = q_{\b}(x^{(r)};t_0)$ is defined similarly to $q^-_{\Bbe}$
under the notation of (5.11.4). 
Then the coefficient of $\Bu^{\Bla}$ in $\Phi_{-}(\Bu)$ is equal to that of $\Bu^{\Bla}$ in   
\begin{equation*}
\frac{\prod_{\substack{i < j, (k,i) \in \vD_1 }}
         (1 - (u_i^{(k)})\iv u_j^{(k)})}
       {\prod_{\substack{(k,i) < (k-1,\ell) \\ (k,i), (k-1,i) \in \vD_1}}
                 (1 - t_{k-1}(u_i^{(k)})\iv u^{(k - 1)}_{\ell})}
            \prod_{\substack{i < j \\ (r,i) \in \vD_0 }}
         \frac{1 - (u_i^{(r)})\iv u_j^{(r)}}{1 - t_0(u_i^{(r)})\iv u^{(r)}_j}
      \sum_{\Bbe}q^-_{\Bbe}\Bu^{\Bbe},
\end{equation*}
where $q^-_{\Bbe}(x;\Bt)$ should be replaced by 
$q_{\b^{(r)}}(x^{(r)};t_0)$ if $\Bbe = (-,\dots, -, \b^{(r)})$. 
Thus the coefficient of $\Bu^{\Bla}$ is equal to 
\begin{equation*}
\biggl(\prod_{\substack{\nu = (k,i) \\ \nu' = (k',j) \\ \nu < \nu'}}
               \frac{\prod_{\substack{k  = k', \nu \in \vD_1}}
            (1 - R_{\nu,\nu'})}
       {\prod_{\substack{k' = k -1 \\ (k,i) \in \vD_1 \\ (k-1,i) \in \vD_1 }}
            (1  - t_{k'}R_{\nu,\nu'})}
         \prod_{\substack{\nu < \nu' \\ b(\nu) = b(\nu') \\ \nu \in \vD_0}}
            \frac{1 - R_{\nu,\nu'}}{1  - t_0R_{\nu,\nu'}}\biggr)q_{\Bla}^-.                 
\end{equation*}
Hence (5.14.1) holds.  The ``$+$'' case is dealt with similarly, and we obtain (5.14.2).  
It follows that $\wt Q^{\pm}_{\Bla}$ is a sum of $R q^{\pm}_{\Bla} = q^{\pm}_{R\Bla}$ 
for various raising operators $R$.  $\Bmu' = R\Bla$ is an $r$-composition if $R q_{\Bla} \ne 0$, 
and in that case we have $\Bmu' \ge \Bla$.  Let $\Bmu$ be the $r$-partition obtained from $\Bmu'$ 
by permuting the parts of $\Bmu'$.  Then $\Bmu \ge \Bmu'$, and so $\Bmu \ge \Bla$.  
We have $R q^{\pm}_{\Bla} =  q^{\pm}_{\Bmu}$.  Thus $\wt Q^{\pm}_{\Bla}$ 
can be written as a linear combination 
of $q^{\pm}_{\Bmu}$ with $\Bmu \ge \Bla$. The equality $\Bmu = \Bla$ holds only when $R = \id$.     
Hence (5.14.3) holds. 
\end{proof}

By using the characterization of Hall-Littlewood functions in Theorem 2.14, 
we have the following result.

\begin{thm}  %%%%  Theorem  5.15
Assume that $\Bla$ satisfies the condition \rm {(A)} in 5.11 in the ``$+$''case. 
Give no assumption in the ``$-$''case.   Then we have
\begin{equation*}
\wt P^{\pm}_{\Bla}(x;\Bt) = P^{\pm}_{\Bla}(x;\Bt)  \qquad 
\wt Q^{\pm}_{\Bla}(x;\Bt) = Q^{\pm}_{\Bla}(x,\Bt). 
\end{equation*}
\end{thm}

\begin{proof}
The second formula follows from Proposition 4.3 and Corollary 5.14.
Then the first formula follows from (4.4.2). 
\end{proof}

\remark{5.16.} \  
In the ``$+$''case, the inductive argument as in the proof of Proposition 5.12
does not work if (A) is not satisfied.  Although the definition of 
$\wt Q_{\Bla}^+$ makes sense even in that case, 
this function does not coincide with $Q_{\Bla}^+$ in general.  
In next section, we discuss the excluded case, and give the closed formula 
for $Q^{+}_{\Bla}$ without assuming (A). 

%%%%
%%%%
%%%%
\par\bigskip\medskip
\section{Closed formula for Hall-Littlewood functions -- ``$+$''case}

\para{6.1.}
Take $\Bla \in \SP_{n,r}$ and consider $\SM = \SM^{(1)}$ associated to $\Bla$. 
Recall that $\nu_0 = (k_0, i_0)$. 
We define a sequence of integers $m_1 \le m_2 \le \cdots \le m_r = i_0$ 
inductively by $m_1 = \ell(\la^{(1)}), m_k = \max (m_{k-1}, \ell(\la^{(k)}))$ 
for $k \ge 2$.
We define a function $I^{(k)}(x;\Bt)$ as follows; 
\begin{align*}
I^{(k)}(x;\Bt) = \prod_{(k,i) \le \nu_0}\prod_{i < j}(x_i^{(k)} - t_{k-1}x^{(k-1)}_j) 
\end{align*}
for $k = 2, \dots, r$, and $I^{(1)} = J_1J_2\cdots J_r$, where 
\begin{align*}
J_a(x;\Bt) = \prod_{i = m_{a -1}+1}^{m_a}\prod_{i \le j}
     (x^{(a)}_i - t_{a-1}\cdots t_2t_1t_rx_j^{(r)})
\end{align*}
if $m_{a-1} < m_a$, and $J_a(x;\Bt) = 1$ if $m_{a-1} = m_a$. 
(By convention, put $m_0 = 0$.)
Moreover for $k = 1, \dots, r$, put 
\begin{equation*}
I_0^{(k)}(x;\Bt) = \prod_{(k,i) > \nu_0}\prod_{i < j} (x^{(k)}_i - t_kx^{(k)}_j).
\end{equation*}
We define a function $R^{\sharp}_{\Bla}(x;\Bt)$ by

\begin{equation*}
R^{\sharp}_{\Bla}(x;\Bt) = \sum_{w \in S^r_m}w\left(
                           \frac{\displaystyle\prod_{1 \le k \le r}(x^{(k)})^{\la^{(k)} -\ve^{(k)}}
                             I^{(k)}(x;\Bt)I^{(k)}_0(x;\Bt)}
              {\displaystyle\prod_{1 \le k \le r}\prod_{i < j}x^{(k)}_i - x^{(k)}_j}\right).
\end{equation*}
Here $\ve^{(k)} = (\ve^{(k)}_1, \dots, \ve^{(k)}_m )$, 
where $\ve^{(k)}_i = 1$ for $m_{k-1} < i \le m_k$ 
and $\ve^{(k)}_i = 0$ otherwise.  
\par
Let $v'_{\Bla}(\Bt)$ be as in (4.1.4).  By a similar discussion as in the 
proof of (4.2.2), we obtain a formula
\begin{equation*}
\tag{6.1.1}
   R^{\sharp}_{\Bla}(x;\Bt)   
     = v'_{\Bla}(\Bt)\sum_{w \in S^r_m /S'_{\Bla}}
      w\left(\frac{\displaystyle\prod_{1 \le k \le r}(x^{(k)})^{\la^{(k)} - \ve^{(k)}}
                I^{(k)}(x;\Bt)}
           {\displaystyle\prod_{(k,i) \le \nu_0}\prod_{\substack{i < j }}
                          x^{(k)}_i - x^{(k)}_j}  \right).      
\end{equation*}
Thus $R^{\sharp}_{\Bla}(x;\Bt)$ is a polynomial in $\BZ[x_{\SM};\Bt]$, and 
is divisible by $v'_{\Bla}(\Bt)$. Moreover, $v'_{\Bla}(\Bt)\iv R^{\sharp}_{\Bla}$
satisfies the stability property for $m \mapsto \infty$.  
We note that $R^{\sharp}_{\Bla}$ satisfies a similar property 
as in Proposition 4.3, namely, 

\begin{prop} %%%%  Prop 6.2
There exist polynomials $u_{\Bla,\Bmu}(\Bt) \in \BZ[\Bt]$ such that
\begin{equation*}
R^{\sharp}_{\Bla}(x;\Bt) = \sum_{\Bmu \le \Bla}u_{\Bla,\Bmu}(\Bt)s_{\Bmu}(x).
\end{equation*}
\end{prop}

\begin{proof}
We use the same notation as in the proof of Proposition 4.3 
(here we consider the ``$+$'' case), in particular, let $w = (w_1, \dots, w_r) \in S_m^r$ 
be as defined there. 
In the ``$+$''-case in 4.3, we modify the definition of $X'_1$ as follows 
($X'_k$ for $k \ne 1$ are unchanged).
$X'_1 = \coprod_{a = 1}^rY_a$, where 
\begin{equation*}
Y_a = \{ (\nu, \nu') \in \SM^2 \mid w(\nu) = (a, i), m_{a -1}+1 \le i \le m_a,
                b(\nu') = r   \}  
\end{equation*}
Let $\SM_k$ (resp. $\SB_k$) be the $k$-th row of $\SM$ (resp. $\SB$). 
We define a subset $\SB'_{r, a}$ of 
$\SM_r$ by 
$w_a(\SB_a) = w_r(\SB'_{r, a})$. 
Put
\begin{align*}
A^+_{1, a} &= \sum_{\substack{(\nu,\nu') \in Y_{a} \\  
        \nu \in \SB_{a}, \nu' \in \SB'_{r, a} }}s_{\nu,\nu'}, \\
B^+_{1,a} &= \sum_{\substack{(\nu,\nu') \in Y_{a} \\
        \nu \in \SB_{a}, \nu' \in \SM_r - \SB'_{r,a}}}s_{\nu,\nu'}. 
\end{align*}
We put $A^+_1 = \coprod_{a=1}^rA^+_{1,a}, B^+_1 = \coprod_{a=1}^rB^+_{1,a}$.
Then a similar argument as in the proof of Proposition 4.3 works by replacing 
$A^+_1, B_1^+$ there by the current version. Thus the proposition follows. 
\end{proof}

\para{6.3.}
Let $\Bla \in \SP_{n,r}$.  
We define a polynomial $f_{\Bla}(\Bt)$ by
\begin{equation*}
\tag{6.3.1}
f_{\Bla}(\Bt) = \prod_{i = 1}^{r-1}t_i^{A_i},
\end{equation*}  
where $A_i = (m -i_0) + (m - i_0+1) + \cdots + (m - m_i-1)$.  
We define a function $Q^{\sharp}_{\Bla}(x;\Bt)$ by 
\begin{equation*}
Q^{\sharp}_{\Bla}(x;\Bt) = v'_{\Bla}(\Bt)\iv f_{\Bla}(\Bt)\iv R^{\sharp}_{\Bla}(x;\Bt). 
\end{equation*}
We show the following result.

\begin{thm}  %%%%   Theorem 6.4
Let $\Bla \in \SP_{n,r}$ be an arbitrary $r$-partition of $n$. 
\begin{enumerate}
\item
$Q^{\sharp}_{\Bla}(x;\Bt)$ coincides with $Q^+_{\Bla}(x;\Bt)$.
\item
$Q^+_{\Bla}(x;\Bt)$ can be expressed as 
\begin{equation*}
Q^+_{\Bla}(x;\Bt) = \prod_{\nu < \nu'}\frac{\prod_{b(\nu) = b(\nu')}1 - R_{\nu,\nu'}}
             {\prod_{b(\nu') = b(\nu) + 1}1 - t_{b(\nu)}R_{\nu,\nu'}}{q}^+_{\Bla}.
\end{equation*}
\item 
$P^+_{\Bla}(x;\Bt) = (1 - t_0)^{-j_0}Q^+_{\Bla}(x;\Bt)$. 
\end{enumerate}
\end{thm}

\para{6.5.}  We prove the theorem in 6.10 after some preliminaries. 
Assume that $\Bla = (-,\la^{(2)}, \dots, \la^{(r)}) \in \SP_{n,r}^1$, and 
put $\Bmu = (\la^{(2)}, \dots, \la^{(r)})$.
We consider the function ${Q'}^{\sharp}_{\Bmu}(x';\Bt')$, defined similarly to 
$Q^{\sharp}_{\Bla}(x;\Bt)$, but by replacing $x$ by 
$x' = (x^{(2)}, \dots, x^{(r)})$ and $\Bt$ by $\Bt' = (t_2, \dots, t_{r-1}, t_1t_r)$.
We have a lemma.

\begin{lem}   %%%%   Lemma 6.6
Assume that $\Bla \in \SP^1_{n,r}$.  
Under the notation as above, we have 
\begin{equation*}
Q^{\sharp}_{\Bla}(x;\Bt) = {Q'}^{\sharp}_{\Bmu}(x';\Bt').
\end{equation*}  
\end{lem}

\begin{proof}
Since $m_1 = 0$, we have $J_1 = 1$. Hence
\begin{align*}
\tag{6.6.1}
& v'_{\Bla}(\Bt)\iv R^{\sharp}_{\Bla}(x;\Bt)  \\ 
   &= \sum_{w' \in S_{\Bm'}/S'_{\Bmu}}w'\biggl\{\sum_{w \in S^{(1)}_m /S'_{\la^{(1)}}}
        w\biggl(\prod_{\substack{i \le i_0 \\ i < j}}\frac{x_i^{(2)} - t_1x^{(1)}_j}
                          {x_i^{(1)} - x^{(1)}_j}\biggr)
             \frac{\prod_{k=2}^r(x^{(k)})^{\la^{(k)} -\ve^{(k)}}
                             I'^{(k)}(x';\Bt')}
             {\prod_{k = 2}^r\prod_{\substack{i \le i_0 \\i < j}}x^{(k)}_i - x^{(k)}_j}\biggr\},
\end{align*}
where $I'^{(k)}(x';\Bt')$ is defined by replacing $r$ by $r-1$ with respect to the variables
$x'$, and $S_{\Bm'} = S_m^{(2)} \times \cdots \times S_m^{(r)}$. 
We note that
\begin{equation*}
\tag{6.6.2}
\sum_{w \in S_m^{(1)}/S'_{\la^{(1)}}}w
   \biggl(\prod_{\substack{i \le i_0 \\ i < j}}\frac{x^{(2)}_i - t_1x^{(1)}_j}
                { x^{(1)}_i - x^{(1)}_j}\biggl)
     =  t_1^{A_1}.
\end{equation*} 
In fact,  let $S_0$ be the subgroup of $S_m$ which stabilizes  
$1, \dots,i_0$. Since $\la^{(1)} = \emptyset$, $S'_{\la^{(1)}} = S_0$. 
Let $X$ be the left hand side of (6.6.2). We use the notation $x^{(1)}_i = y_i, 
x^{(2)}_i = z_i$. $S_m$ acts only on $y_i$ variables.  Then 
\begin{align*}
X &= |S_0|\iv\sum_{w \in S_m }w\biggl\{\prod_{i < j}
           \frac{\prod_{\substack{i \le i_0 }}
            (z_i - t_1y_j)\prod_{\substack{i > i_0 }}
              (y_i - y_j)}{ y_i - y_j} \biggr\}\\
             &= |S_0|\iv\prod_{i < j}(y_i - y_j)\iv
                    \sum_{w \in S_m}\ve(w)w\biggl(\prod_{\substack{i \le i_0 \\ i < j}}
                        (z_i - t_1y_j) \prod_{\substack{ i > i_0 \\ i < j}}
                         ( y_i - y_j)\biggr). 
\end{align*}
Since $\sum_{w \in S_m}\ve(w)w\prod (z_i - t_1y_j) 
              \prod ( y_i - y_j)$
is an alternating polynomial with respect to $y_i$, 
non-zero contribution only comes from the term 
\begin{align*}
\prod_{\substack{i \le i_0 \\ i < j}}&(-t_1y_j) 
        \prod_{\substack{ i> i_0 \\ i < j}}(y_i - y_j)  \\ 
        &= (-t_1)^{A_1} (y_m y_{m-1}\cdots y_{i_0+1})^{i_0}
              (y_{i_0})^{i_0 -1}(y_{i_0-1})^{i_0-2}\cdots 
           y_2 \prod_{\substack{i > i_0 \\ i < j}}
                      (y_i - y_j). 
\end{align*}  
Since the last product can be written as 
\begin{equation*}
\sum_{w' \in S_0}\ve(w')w'
   \bigl((y_{i_0+1})^{m -i_0-1}(y_{i_0+2})^{m -i_0-2}\cdots y_{m -1}\bigr),
\end{equation*}   
we have
\begin{align*}
X &= (-t_1)^{A_1}\prod_{i < j}(y_i - y_j)\iv  \\
        &\times \sum_{w \in S_m}\ve(w)w\bigl((y_{i_0+1})^{m-1}(y_{i_0+2})^{m-2}\cdots 
                   (y_{m-1})^{i_0+1}(y_m)^{i_0}
             (y_{i_0})^{i_0-1}(y_{i_0-1})^{i_0-2}\cdots y_2\bigr)  \\
   &= t_1^{A_1}. 
\end{align*}
Thus (6.6.2) holds. 
\par
By (6.6.1) and (6.6.2), we have 
$v'_{\Bla}(\Bt)\iv R^{\sharp}_{\Bla}(x;\Bt) = 
      t_1^{A_1}v'_{\Bmu}(\Bt')\iv {R'}^{\sharp}_{\Bmu}(x';\Bt')$. 
Since $f_{\Bla}(\Bt) = t_1^{A_1}f_{\Bmu}(\Bt')$, we obtain the lemma. 
\end{proof}

\para{6.7.}
We consider the special case where $m_1 = \dots = m_{r-1} = 0$ 
and $m_r = i_0$. In this case, 
\begin{align*}
I^{(1)} = J_r &= \prod_{1 \le i \le i_0}\prod_{i \le j}(x_i^{(r)} - t_0x_j^{(r)})  \\
              &= (1 - t_0)^{i_0} \prod_{1 \le i \le i_0}x^{(r)}_i\prod_{i < j}
                    (x_i^{(r)} - t_0x_j^{(r)}),
\end{align*}
where $t_0 = t_1\cdots t_r$. 
We write $\Bla = (-; \cdots; -; \mu)$ with $\mu \in \SP_n$. 
By a similar computation as in the proof of 
Lemma 6.6 (note that $f_{\Bla}(\Bt) = 1$), we have 
\begin{equation*}
Q^{\sharp}_{\Bla}(x;\Bt) = v'_{\Bla}(\Bt)\iv R^{\sharp}_{\Bla}(x;\Bt) = 
       (1 - t_0)^{i_0} v_{\la'_r}(t_0)\iv
              R_{\mu}(x^{(r)}; t_0),
\end{equation*}
where $R_{\mu}$ is the function defined in [M, III, 1]. 
Under the notation in [M, III, 2], we have 
\begin{equation*}
R_{\mu}(x^{(r)};t_0) = v_{\mu}(t_0)P_{\mu}(x^{(r)};t_0)
    = v_{\mu}(t_0)b_{\mu}(t_0)\iv Q_{\mu}(x^{(r)};t_0),
\end{equation*}  
where $P_{\mu}, Q_{\mu}$ are classical Hall-Littlewood functions associated 
to the partition $\mu$. 
Since $v_{\mu}(t_0) = v'_{\mu}(t_0)b_{\mu}(t_0)/(1 - t_0)^{i_0}$ 
(see [M, III,2], note that $v_{\Bla'_r}(t_0) = v_{m - i_0}(t_0)$), we have 
the following.
\begin{lem}  %%%%  Lemma 6.8
Assume that $\Bla \in \SP_{n,r}^{r-1}$.  Then 
\begin{equation*}
\tag{6.8.1}
Q^{\sharp}_{\Bla}(x;\Bt) = Q_{\la^{(r)}}(x^{(r)}; t_1\cdots t_r).  
\end{equation*}
\end{lem}

Next we show the following proposition.
%%%%
%%%%
\begin{prop} %%% Prop. 6.9
Assume that Theorem 6.4 holds for $r-1$.  Then for any $\Bla \in \SP^1_{n,r}$, 
we have
\begin{equation*}
\tag{6.9.1}
Q^+_{\Bla}(x;\Bt) = 
       \prod_{\nu < \nu'}\frac{\prod_{b(\nu) = b(\nu')}1 - R_{\nu,\nu'}}
             {\prod_{b(\nu') = b(\nu) + 1}1 - t_{b(\nu)}R_{\nu,\nu'}}{q}^+_{\Bla}.  
\end{equation*}
\end{prop}

\begin{proof}
Write $\Bla = (-, \la^{(2)}, \dots, \la^{(r)}) \in \SP^1_{n,r}$ 
and $\Bla' = (\la^{(2)}, \dots, \la^{(r)}) \in \SP_{n,r-1}$.
By Proposition 3.4, we have $Q^+_{\Bla}(x;\Bt) = {Q'}^+_{\Bla'}(x', \Bt')$ with 
$\Bt' = (t_2, \dots, t_{r-1}, t_1t_r)$. 
By applying Theorem 6.4 (ii) for ${Q'}^+_{\Bla'}$, 
we have an expression 
\begin{equation*}
\tag{6.9.2}
{Q'}^+_{\Bla'}(x';\Bt') = \prod_{\nu < \nu'}\frac{\prod_{b(\nu) = b(\nu')}1 - R'_{\nu,\nu'}}
             {\prod_{b(\nu') = b(\nu) + 1}1 - t'_{b(\nu)}R'_{\nu,\nu'}}{q'}^+_{\Bla'}.  
\end{equation*}
where $t'_{b(\nu)}$ is defined with respect to $\Bt'$, and  
$R'_{\nu,\nu'}$ is the raising operator with 
respect to $\SM' = \{ (k,i) \mid 2 \le k \le r, 1 \le i \le m\}$. 
In particular $R'_{\nu,\nu'} = R_{\nu,\nu'}$ if $b(\nu) = b(\nu')$ 
or $b(\nu') = b(\nu)+1$ with $b(\nu) \ne r$. 
Assume that $b(\nu) = r$ and $b(\nu') = 2$.  
Write $\nu = (r, i), \nu' = (2, j)$.  In the computation below, we omit the sign ``$+$''.
By Lemma~3.2, 
${q'}^{(2)}_{j}(x; t_1t_r) = 
   \sum_{j_1 + j_2 = j}t_1^{j_1}q^{(2)}_{j_2}(x;t_1)q^{(1)}_{j_1}(x;t_r)$. 
For $a \ge 1$, we denote by $(tR'_{\nu,\nu'})^{[a]}$ the operator 
$(tR'_{\nu + a - 1,\nu'-a + 1})\cdots (tR'_{\nu+1, \nu-1})(tR'_{\nu, \nu'})$.
Then 
the action of $(t_1t_rR'_{\nu,\nu'})^{[a]}$ on ${q'}^{(2)}_{j}{q'}^{(r)}_{i}$ can be written as 
\begin{align*}
(t_1t_rR'_{\nu,\nu'})^{[a]}\bigl({q'}^{(2)}_{j}{q'}^{(r)}_{i}\bigr) 
    &= (t_1t_r)^a{q'}^{(2)}_{j-a}{q'}^{(r)}_{i+a} \\
    &= (t_1t_r)^a\sum_{j' + j'' = j-a}(t_1)^{j''}q^{(2)}_{j'}q^{(1)}_{j''}q^{(r)}_{i+a} \\
    &= \sum_{j_1 + j_2 = j}\sum _{0 \le b \le j_2} 
         (t_1)^{j_1 + b}t_r^aq^{(2)}_{j_2 -b}q^{(1)}_{j_1 - a + b}q^{(r)}_{i + a} \\
    &= \sum_{j_1 + j_2 = j}\sum_{0 \le b \le j_2}(t_1R^{(1)}_{j_1 - a,j_2})^{[b]}
               (t_rR^{(r)}_{i,j_1})^{[a]}
                        \bigl(t_1^{j_1}q^{(2)}_{j_2}q^{(1)}_{j_1}q^{(r)}_i\bigr) \\
    &= \sum_{j_1 + j_2 = j}\sum_{0 \le b \le j_2}(t_1R^{(1)}_{j_1 - a,j_2})^{[b]}
               (t_rR^{(r)}_{i,j_1})^{[a]}
                        (t_1R^{(1)}_{0,j})^{[j_1]}q^{(2)}_{j}q^{(r)}_i, 
\end{align*}
where $R^{(r)}_{i,j} = R_{(r,i), (1,j)}$ and $R^{(1)}_{j_1, j_2} = R_{(1,j_1), (2, j_2)}$. 
We understand that $R^{(1)}_{0,j}$ is the raising operator which sends 
$q^{(1)}_0 = 1$ to $q^{(1)}_1$ and $q^{(2)}_j$ to $q^{(2)}_{j-1}$. 
It follows that for any $\Bla \in \SP^1_{n,r}$, 
\begin{align*}
&\prod_{\substack{\nu < \nu' \\
       b(\nu) = r, b(\nu') = 2}}(1 - t_1t_rR'_{\nu,\nu'})\iv q'_{\Bla'}(x;\Bt') \\ 
  &= \prod_{\substack{\nu < \nu' \\b(\nu) = b(\nu') = 1}}(1 - R_{\nu,\nu'})
       \prod_{\substack{\nu < \nu' \\b(\nu) = 1 \\ b(\nu') = 2}} 
             (1 - t_1R_{\nu, \nu'})\iv
    \prod_{\substack{\nu < \nu' \\ b(\nu) = r \\ b(\nu') = 1}}
             (1 - t_rR_{\nu, \nu'})\iv q_{\Bla}(x;\Bt).
\end{align*}
Note that since $\la^{(1)} = \emptyset$, the operator 
$\prod_{b(\nu) = b(\nu') = 1}(1 - R_{\nu,\nu'})$ acts trivially. 
The proposition follows from this formula.
\end{proof}

\para{6.10.}
We are now ready to prove Theorem 6.4. Assume that $\la^{(1)} \ne \emptyset$.  
Let $\SM = \SM^{(1)}$ in the notation of 5.1, and consider $\SM ^{(a)}$ for 
$a = 1, \dots, r$. Then a similar formula as in Lemma 5.5
holds for $Q^{\sharp}_{\Bla}$. Repeating this procedure for $a = 1$ to $r$, one can replace
$\SM$ by $\SM'$ which is defined by replacing $m$ by $m-1$.  
By induction, we may assume that the statements (i), (ii) of the theorem hold 
for the corresponding function ${Q'}^{\sharp}_{\Bmu}$ on $\SM'$. 
In particular, ${Q'}^{\sharp}_{\Bmu}$ coincides with ${Q'}^+_{\Bmu}$, and 
has an expression in terms of raising operators
(assertion (ii) of the theorem).  This is equivalent to saying that 
${Q'}^{\sharp}_{\Bmu}$ can be expressed as a coefficient of ${\Bu'}^{\Bmu}$ in the 
function $\Phi_+(\Bu')$ given in Proposition~5.12.  
Then a similar argument as in the proof of Proposition 5.12 works for $Q^{\sharp}_{\Bla}$, 
and $Q^{\sharp}_{\Bla}$ can be expressed as the coefficient of $\Bu^{\Bla}$ in 
$\Phi_+(\Bu)$, in other words, $Q^{\sharp}_{\Bla}$ has an expression in terms of 
raising operators as in (ii) of the theorem (see Corollary~5.14).  
Since $Q^{\sharp}_{\Bla}$ satisfies 
the condition for the expansion by Schur functions (Lemma~6.2), we see that 
$Q^{\sharp}_{\Bla} = Q^+_{\Bla}$ by Theorem 2.14.  
Hence (i), (ii) holds in this case.
\par 
Next assume that $\la^{(1)} = \emptyset$, i.e., $\Bla \in \SP_{n,r}^1$.  
In this case, by Proposition 3.4, we have 
$Q^+_{\Bla}(x;\Bt) = {Q'}^+_{\Bmu}(x',\Bt')$ with 
$\Bmu = (\la^{(2)}, \dots, \la^{(r)})$ and $\Bt' = (t_2, \dots, t_{r-1}, t_1t_r)$. 
We also have $Q^{\sharp}_{\Bla}(x;\Bt) = {Q'}^{\sharp}_{\Bmu}(x';\Bt')$ by 
Lemma 6.6. 
By induction, we assume that the theorem holds for ${Q'}^+_{\Bmu}$.  
Hence ${Q'}^{\sharp}_{\Bmu}(x',\Bt') = {Q'}^+_{\Bmu}(x';\Bt')$.
This implies that $Q^{\sharp}_{\Bla}(x;\Bt) = Q^+_{\Bla}(x;\Bt)$, which proves (i).  
Then (ii) holds for $Q^+_{\Bla}(x;\Bt)$ by Proposition~6.9. 
\par 
It remains to consider the case where $\Bla \in \SP^{r-1}_{n,r}$, namely 
$\Bla = (-, \dots, -,\la^{(r)})$.  In this case, by Theorem 3.5 and Lemma 6.8, we have
$Q^{\sharp}_{\Bla}(x;\Bt) = Q^+_{\Bla}(x,\Bt) = Q_{\la^{(r)}}(x^{(r)}; t_1\cdots t_r)$.
Hence (i) holds. 
By [M, III, $(2.15')$], $Q_{\la^{(r)}}$ has an expression by raising operators. 
Then by a similar argument as in the proof of Proposition 6.9, one sees that 
$Q^+_{\Bla}$ has an expression by raising operators as in (ii) of the theorem.
\par
Finally, we show (iii).  
We already know by (4.2.5) and Theorem 5.15 
that $Q^-_{\Bla}= (1 - t_0)^{j_0}P^-_{\Bla}$. 
Then (iii) follows from (2.13.1).  The theorem is proved.   
\par\medskip
Combining Theorem 5.15 and Theorem 6.4, we have the following result.

\begin{thm}  %%%%   Theorem 6.11
Let $P^{\pm}_{\Bla}(x;\Bt), Q^{\pm}_{\Bla}(x;\Bt) \in \Xi^n_{\BQ}(\Bt)$ 
be the Hall-Littlewood functions. 
\begin{enumerate}
\item 
$P^{\pm}_{\Bla}(x;\Bt), Q^{\pm}_{\Bla}(x;\Bt) \in \Xi^n[\Bt]$.
\item 
$Q^{\pm}_{\Bla}(x;\Bt) = (1-t_1\cdots t_r)^{j_0}P^{\pm}_{\Bla}(x;\Bt)$, where 
$j_0$ is as in 4.2.
\item
$P^{\pm}_{\Bla}(x;\Bt)$ and $Q^{\pm}_{\Bla}(x;\Bt)$ are characterized by the property 
as in Theorem 2.14, but the total order $\lve$ can be replaced by the dominance order 
$\le $, and the coefficients $c_{\Bla,\Bmu}(\Bt), u_{\Bla,\Bmu}(\Bt) \in \BZ[\Bt]$. 
In particular, $P^{\pm}_{\Bla}, Q^{\pm}_{\Bla}$ are determined independently from 
the choice of the total order.
\item
$\{ P^{\pm}_{\Bla} \mid \Bla \in \SP_{n,r}\}$ give rise to $\BZ[\Bt]$-bases 
of $\Xi^n[\Bt]$. We have $K^{\pm}_{\Bla,\Bmu}(\Bt) \in \BZ[\Bt]$. 
\end{enumerate}
\end{thm}

\begin{proof}
(ii) follows from Theorem 5.15 and Theorem 6.4.
By (4.2.4) and (4.2.5), $\wt P^-_{\Bla}(x;\Bt), \wt Q^-_{\Bla}(x;\Bt) \in \BZ[x_{\SM};\Bt]$.  
Thus $P^-_{\Bla}, Q^-_{\Bla} \in \Xi^n[\Bt]$ by Theorem 5.15. 
By Theorem 6.4, $Q^+_{\Bla}$ has
an expression by raising operators.  Thus $Q^+_{\Bla} \in \Xi^n[\Bt]$.  
Hence (i) holds.  Then (iii) follows from Proposition 4.3 and Corollary 5.14 in the 
``$-$''case, and from Proposition 6.2 and Theorem 6.4 in the ``$+$''case. (iv) follows from (iii).   
\end{proof}

%%%%
%%%%
%%%%
\par\medskip
\section{A conjecture of Finkelberg-Ionov}

\para{7.1}
For two basis $u = \{u_{\Bla}\}, v = \{v_{\Bmu} \}$ on the $\BQ(\Bt)$-space $\Xi_{\BQ}(\Bt)$, 
we denote by $M(u,v) = (M_{\Bla\Bmu})$ the transition matrix between $u$ and $v$ 
as in the proof of Theorem 2.14., 
Recall the non-degenerate bilinear form $\lp \ , \ \rp$ on $\Xi^n_{\BQ}(\Bt)$ 
introduced in 2.10, which satisfies the properties
 
\begin{align*}
\lp q^{+}_{\Bla}, m_{\Bmu}\rp &= \lp m_{\Bla}, q^{-}_{\Bmu}\rp = \d_{\Bla,\Bmu}, \\
\lp P^+_{\Bla}, Q^-_{\Bmu}\rp &= \lp Q^+_{\Bla}, P^-_{\Bmu}\rp = \d_{\Bla,\Bmu}. \\
\end{align*} 
For a matrix $M$, let $M^*$ be the matrix $^t M\iv$.
We denote the basis $\{P^{\pm}_{\Bla}| \Bla \in \SP_{n,r}\}$ of $\Xi^n_{\BQ}(\Bt)$ by
$P^{\pm}$, and similarly define $Q^{\pm}, s, m$, with respect to 
$Q^{\pm}_{\Bla}, s_{\Bla}, m_{\Bla}$, respectively.  
Then we have
\begin{equation*}
M(Q^{\pm}, q_{\pm}) = M(P^{\mp}, m)^* = M(P^{\mp}, s)^*M(s,m)^*.
\end{equation*}
Since $M(s, P^{\mp}) = K(\Bt)^{\mp} = (K^{\mp}_{\Bla,\Bmu}(\Bt))$,  we have
$M(P^{\mp}, s)^* = {}^t(K(\Bt)^{\mp})$. 
The Kostka number $K_{\Bla,\Bmu}$ for $\Bla, \Bmu \in \SP_{n,r}$ is defined by 
$K_{\Bla,\Bmu} = \prod_iK_{\la^{(i)},\mu^{(i)}}$ if $|\la^{(i)}| = |\mu^{(i)}|$
for each $i$, and $K_{\Bla,\Bmu} = 0$ otherwise. Put $K = (K_{\Bla,\Bmu})$. 
We know that $M(s,m)^* = K^* = M(s, h)$, and by [M, I, $(3.4')$]
\begin{equation*}
s_{\Bla} = \prod_{\nu < \nu'}\prod_{b(\nu) = b(\nu')}(1 - R_{\Bnu,\Bnu'})h_{\Bla},
\end{equation*} 
where $h_{\Bla}$ is the complete symmetric function 
defined similarly to $s_{\Bla}, m_{\Bla}$.
Hence the operation $\prod_{\nu, \nu'}(1 - R_{\nu,\nu'})$ on the basis 
$\{ h_{\Bmu} \}$ corresponds to the matrix operation  $K^*$. 
Moreover, the matrix operation $M(P^{\mp},s)^*$ on the basis $\{ s_{\Bmu} \}$ 
coincides with the matrix operation $M(Q^{\pm},q_{\pm})$ on the basis $\{ h_{\Bmu} \}$. 
In particular, if we write $M(Q^{\pm},q_{\pm}) = (c^{\pm}_{\Bla\Bmu})$, 
we have

\begin{equation*}
\sum_{\Bmu}K^{\pm}_{\Bmu,\Bla}(\Bt)s_{\Bmu} = \sum_{\Bmu}c^{\mp}_{\Bla\Bmu}h_{\Bmu}
\end{equation*}
for each $\Bla$. 
By Corollary 5.14 and Theorem 6.4, we have an expression 
for $Q^{\pm}_{\Bla}$ such as (5.14.1), (5.14.2). 
Hence, for a fixed $\Bla$,
we have 
\begin{align*}
\tag{7.1.1}
\sum_{\Bmu}c^{+}_{\Bla\Bmu}h_{\Bmu} &= 
    \prod_{\substack{\nu < \nu' }}
        \frac{\prod_{b(\nu) = b(\nu')}(1 - R_{\nu,\nu'})}
       {\prod_{b(\nu') = b(\nu) + 1}(1  - t_{b(\nu)}R_{\nu,\nu'})}h_{\Bla} \\
             &= \prod_{\substack{\nu < \nu' }}
                    \prod_{b(\nu') = b(\nu) + 1}(1  - t_{b(\nu)}R_{\nu,\nu'})\iv s_{\Bla}, 
  \\ \\   
\tag{7.1.2}
\sum_{\Bmu}c^-_{\Bla\Bmu}h_{\Bmu} &= 
       \biggl(\prod_{\substack{\nu = (k,i) \\ \nu' = (k',j) \\ \nu < \nu'}}
                  \frac{\prod_{k = k', \nu \in \vD_1}(1 - R_{\nu,\nu'})}
                    {\prod_{\substack{k' = k-1  \\ (k,i) \in \vD_1 \\ (k',i) \in \vD_1}}
                           (1 - t_{k'}R_{\nu,\nu'})}
         \prod_{\substack{\nu < \nu' \\b(\nu) = b(\nu') \\ \nu \in \vD_0}}
             \frac{1 - R_{\nu,\nu'}}{1 - t_0R_{\nu,\nu'}}\biggr)h_{\Bla} \\   
          &= \prod_{\substack{\nu = (k,i) \\ \nu' = (k',j) \\ \nu < \nu'}}
               \prod_{\substack{k' = k-1  \\ (k,i) \in \vD_1 \\ (k',i) \in \vD_1}}
                   (1 - t_{k'}R_{\nu,\nu'})\iv          
           \prod_{\substack{\nu < \nu' \\b(\nu) = b(\nu') \\ \nu \in \vD_0}}
             (1 - t_0R_{\nu,\nu'})\iv s_{\Bla}.
\end{align*}
It follows that the right hand sides of (7.1.1) and (7.1.2) coincide with 
$\sum_{\Bmu}K^{\mp}_{\Bmu,\Bla}(\Bt)s_{\Bmu}$.  

\para{7.2.}
We keep the notation for $\SM = \SM^{(1)}$. 
Put $M = rm$.  The $n$-function $n(\xi)$ for $\xi \in \SP_M$ can be 
extended to any composition $\xi = (\xi_i)_{1 \le i \le M} \in \BZ^M_{\ge 0}$ 
by $n(\xi) = \sum_{i=1}^M(i-1)\xi_i$.   
Recall the $a$-function on $\SP_{n,r}$ ([S1]),
\begin{equation*}
a(\Bla) = r\cdot n(\Bla) + |\la^{(2)}| + 2|\la^{(3)}| + \cdots + (r-1)|\la^{(r)}|,
\end{equation*}  
where $n(\Bla) = n(\la^{(1)}) + \cdots + n(\la^{(r)})$. 
Let $c(\Bla)$ be the composition of $M$ associated to $\Bla$ as in 2.11.
We note that
\begin{equation*}
\tag{7.2.1}
n(c(\Bla)) = a(\Bla).
\end{equation*}
In fact 
\begin{align*}
n(c(\Bla)) &= \sum_{k=1}^r\sum_{i=1}^m \bigl((i-1)r + (k-1)\bigr)\la^{(k)}_i \\
           &= \sum_{k = 1}^rr\cdot n(\la^{(k)}) + \sum_{k=1}^r(k-1)|\la^{(k)}| \\
           &= a(\Bla).  
\end{align*}
\par
Let $\lp \ , \ \rp$ be the standard inner product on $\BZ^M$, and put 
$\d = (M-1, M-2, \dots, 0) \in \BZ^M$. Take $\Bla, \Bmu \in \SP_{n,r}$.  Then we have
\begin{equation*}
\tag{7.2.2}
\lp c(\Bla) + \d, \d \rp - \lp c(\Bmu)+\d, \d\rp = a(\Bmu) - a(\Bla).
\end{equation*}
In fact, if we put $\BM  = (M -1, M-1, \dots, M-1) \in \BZ^M$, then  
$\lp c(\Bla), \BM - \d\rp = n(c(\Bla)) = a(\Bla)$. Also we have 
$\lp c(\Bla), \BM \rp = \lp c(\Bmu), \BM\rp$ since $c(\Bla), c(\Bmu)$ are 
compositions of $M$.
Hence
\begin{align*}
a(\Bmu) - a(\Bla) &= \lp c(\Bmu), \BM - \d\rp - \lp c(\Bla), \BM - \d\rp \\ 
                  &=  \lp c(\Bla), \d\rp - \lp c(\Bmu), \d\rp.
\end{align*}
Thus (7.2.2) holds. 
\par
Let $\ve_1, \dots, \ve_M$ be the standard basis of $\BZ^M$. We denote by 
$R^+$ the set of positive roots of type $A_{M-1}$, namely 
$R^+ = \{ \ve_i - \ve_j \mid 1 \le i < j\le M\}$.  
We identify $\BZ^M$ with $\BZ^{\SM}$ by the given total order, and denote 
$\xi = (\xi_i)\in \BZ^M$ as $\xi = (\xi_{\nu})$ with $\nu \in \SM$. 
For any $\xi = (\xi_i) \in \BZ^M$
such that $\sum \xi_i = 0$, we define a function $L_{\pm}(\xi;\Bt)$ by 
\begin{equation*}
\tag{7.2.3}
L_{\pm}(\xi;\Bt) = \sum_{(m_{\g})}\prod_{\g}t_{b(\nu) - c}^{m_{\g}},
\end{equation*}   
where $c$ is as in (2.5.3), and 
$(m_{\g})$ runs over all the choices such that $\xi = \sum_{\g \in R^+}m_{\g}\g$ 
with $m_{\g} \ge 0$, and that $\g = \ve_{\nu} - \ve_{\nu'}$ for $\nu < \nu'$ with the condition 
\begin{equation*}
\tag{7.2.4}
b(\nu') = b(\nu) \mp 1.  
\end{equation*}
Note that $L_{\pm}(\xi;\Bt) \ne 0$ only when $\xi = \sum_i\e_i(\ve_i - \ve_{i+1})$ 
with $\e_i \ge 0$ satisfying the condition (7.2.4). We have $L_-(\xi;\Bt)$ is monic
of degree $\sum_i\e_i = \lp \xi, \d \rp$ (see [M, III, 6, Ex. 4]), and 
$\deg L_+(\xi;\Bt) < \lp \xi, \d\rp$ if $r \ge 3$.   
\par
In the ``$+$''-case, we define a function $L^{\Bmu}_+(\xi;\Bt)$ (depending on the
choice of $\Bmu$) by 
\begin{equation*}
\tag{7.2.5}
L^{\Bmu}_+(\xi;\Bt) = \sum_{(m_{\g})}\biggl(
      \prod_{\g \in A_1}t_{k-1}^{m_{\g}}\prod_{\g \in A_0}t_0^{m_{\g}}\biggr),
\end{equation*}   
where $(m_{\g})$ runs over all the choices such that $\xi = \sum_{\g \in R^+}m_{\g}\g$ 
with $m_{\g} \ge 0$, and that $\g = \ve_{\nu} - \ve_{\nu'}$ for $\nu < \nu'$ 
with the condition 
\begin{align*}
\tag{7.2.6}
A_1 &: \ \nu = (k,i) \in \vD_1, (k-1,i) \in \vD_1, \nu' = (k-1,j), \\
A_0 &: \ \nu \in \vD_0, b(\nu) = b(\nu') = r.  
\end{align*}
(here $\vD = \vD(\Bmu)$ is the subset of $\SM$ determined as in 5.11).

We have the following result.

\begin{thm}  %%%%   Theorem 7.3.
Let $\Bla, \Bmu \in \SP_{n,r}$.  Under the natural embedding $S_m^r \subset S_M$,
 we have
\begin{align*}
\tag{7.3.1}
K^{-}_{\Bla,\Bmu}(\Bt) &= \sum_{w \in S^r_m}\ve(w)
             L_{-}\bigl(w\iv (c(\Bla) + \d) - (c(\Bmu) + \d); \Bt \bigr), \\
\tag{7.3.2}
K^{+}_{\Bla,\Bmu}(\Bt) &= \sum_{w \in S^r_m}\ve(w)
             L^{\Bmu}_{+}\bigl(w\iv (c(\Bla) + \d) - (c(\Bmu) + \d); \Bt \bigr).
\end{align*} 
In particular, $K^-_{\Bla,\Bmu}(\Bt)$  is monic of degree $a(\Bmu) - a(\Bla)$, 
and $\deg K^+_{\Bla,\Bmu} < a(\Bmu) - a(\Bla)$ if $r \ge 3$. 
\end{thm}

\begin{proof}
By (7.1.1), $K^{-}_{\Bla,\Bmu}(\Bt)$ is the coefficient of $s_{\Bla}$ in 
\begin{equation*}
\prod_{\substack{\nu < \nu'}}
         \prod_{b(\nu') = b(\nu) + 1}
    (1 + t_{b(\nu)}R_{\nu,\nu'} + t_{b(\nu)}^2R^2_{\nu,\nu'} + \cdots )s_{\Bmu}, 
\end{equation*}
hence is the coefficient of $x^{\Bla + \Bdel_1}$ in 
\begin{equation*}
\sum_{w \in S^r_m}\ve(w)\sum_{(m_{\g})}\prod_{\g}t_{b(\nu)}^{m_{\g}}
                     x^{w(\Bmu + \Bdel_1 + \sum m_{\g})}, 
\end{equation*}
where $\Bdel_1 = (\d^{(1)}, \dots, \d^{(r)})$ with $\d^{(k)} = (m-1, \dots, 0)$  
and $(m_{\g})$ are given as in (7.2.3) and (7.2.5).  
We now consider the change of variables $\{ x^{(k)}_i \mid 1 \le k \le r, 1 \le i \le m\}$ 
to $\{ y_j \mid 1 \le j \le M \}$ by the assignment $x^{(k)}_i \mapsto y_{(i-1)r + k}$. 
Then the above coefficient coincides with the coefficient of $y^{c(\Bla) + \d}$ in 
\begin{equation*}
\sum_{w \in S^r_m}\ve(w)\sum_{(m_{\g})}\prod_{\g}t_{b(\nu)}^{m_{\g}}
                  y^{w(c(\Bmu) + \d + \sum m_{\g})}.
\end{equation*}
This proves (7.3.1).  
(7.3.2) is proved in a similar way by using (7.1.2). 
\par
We have
\begin{align*}
\lp w\iv(c(\Bla) + \d) - (c(\Bmu) + \d), \d \rp 
           &= \lp c(\Bla) + \d, w(\d)\rp - \lp c(\Bmu) + \d, \d\rp  \\
           &\le \lp c(\Bla) + \d, \d\rp - \lp c(\Bmu) + \d, \d\rp  \\
           &= a(\Bmu) - a(\Bla).
\end{align*}
The last step follows from  (7.2.2).
The equality holds only when $w = 1$.
Note that $\Bla$ is obtained from $\Bmu$ by applying the raising operator 
$R_{\nu,\nu'}$ with $b(\nu') = b(\nu) + 1$, $c(\Bla) - c(\Bmu)$ can be written 
as a sum of $\g \in R^+$ satisfying (7.2.4).  Hence $L_-(c(\Bla)-c(\Bmu))$ is 
monic of degree $a(\Bmu) - a(\Bla)$, and the same is true for $K^-_{\Bla,\Bmu}(\Bt)$.    
\par
Finally consider the degree of $K^+_{\Bla,\Bmu}(\Bt)$.  In this case,  
for $\nu = (r, i), \nu' = (r, i+1)$, $\g = \ve_{\nu} - \ve_{\nu'}$ can be written as 
\begin{equation*}
\g = (\ve_{(r,i)} - \ve_{(1,i+1)}) + (\ve_{(1,i+1)} - \ve_{(2, i+1)}) 
          + \cdots + (\ve_{(r-1, i+1)} - \ve_{(r,i+1)}).
\end{equation*}
Then the computation of $L^{\Bmu}_+(\xi;\Bt)$, which involves the monomials with respect to 
$t_1, \dots, t_r$ and $t_0$, is reduced to the computation of $L_+(\xi;\Bt)$, which 
involves only $t_1, \dots, t_r$.  Since $\deg L_+(\xi,\Bt) < \lp \xi, \d\rp$ (for $r \ge 3$), 
the assertion holds.  
\end{proof}

\para{7.4.}
In the ``$+$'' case, we consider a special situation where 
the function $L^{\Bmu}_+$ can be described easily. 
We put the  following condition for $\Bmu \in \SP_{n,r}$;
\par\medskip\noindent
(B) $\ell(\la^{(k)}) = i_0$ for $k = 1, \dots, r$. 
\par\medskip
If $\Bmu$ satisfies the condition (B), then 
$\vD_0 = \emptyset$ and $\vD_1 = \SM$ for $\vD = \vD(\Bmu)$. 
In this case, $L^{\Bmu}_+(\xi;\Bt)$ coincides with $L_+(\xi;\Bt)$.    
Hence as a corollary of Theorem 7.3 (ii), we have the following.
%%%%
%%%%
\begin{cor}   %%%%   Cor. 7.5
Assume that $\Bmu$ satisfies the condition \rm{(B)}.  Then we have
\begin{equation*}
K^+_{\Bla,\Bmu}(\Bt) = \sum_{w \in S^r_m}\ve(w)L_+ 
         \bigl(w\iv(c(\Bla) +\d) - (c(\Bmu) + \d);\Bt\bigl).
\end{equation*}
\end{cor}

\para{7.6.}
In [FI], Finkelberg and Ionov defined the multi-variable 
Kostka polynomials $K_{\Bla,\Bmu}(\Bt)$ 
by using the Lusztig's partition function ([L1]) 
as defined in (7.2.3) for ``$-$''case, which 
is exactly the formula in the right hand side of (7.3.1).  They conjectured 
(in the case where $t = t_1 = \dots =  t_r$) that $K_{\Bla,\Bmu}(t)$ coincides with 
our $K_{\Bla,\Bmu}^-(t)$. Theorem 7.3 gives an affirmative answer to their conjecture
(for the multi-variable case).    
Following [FI], we say that $\Bmu \in \SP_{n,r}$ is regular if 
$\mu_1^{(k)} > \mu_2^{(k)} > \cdots > \mu^{(k)}_m$ for $k = 1, \dots, r$. 
They proved in [FI], in the case where $\Bmu$ is regular,  
that $K_{\Bla,\Bmu}(\Bt) \in \BZ_{\ge 0}[\Bt]$ by making use
of the higher cohomology vanishing of a certain vector bundle over 
the flag variety of $(GL_m)^r$. Recently Hu [H] proved the higher cohomology vanishing 
for arbitrary $\Bmu$, hence the positivity property of $K_{\Bla,\Bmu}(\Bt)$ now holds 
without any restriction. 
\par
Combined with Theorem 7.3, we have

\begin{prop}  %%%%   Prop 7.7.
$K^-_{\Bla,\Bmu}(\Bt) \in \BZ_{\ge 0}[\Bt]$. 
\end{prop}

\remark{7.8.} The last statement in Theorem 7.3 and Proposition 7.7 give 
an answer to the conjecture proposed in [S1, Conjecture 5.5] at least for the 
``$-$''case. 

\para{7.9.} 
Let $\Bth = (\th^{(1)}, \dots, \th^{(r)})$ be an $r$-partition, where 
$\th^{(k)} = (\th_1, \th_2, \dots, \th_m)$ for 
$k = 1, \dots, r$ (independent of $k$). Let $\Bla, \Bmu \in \SP_{n,r}$.  
Then $\Bla + \Bth, \Bmu + \Bth \in \SP_{n',r}$ for some $n'$. 
As a corollary of Theorem 7.3 and Corollary 7.5, we have the following 
result, which was conjectured by Finkelberg (for the ``$-$''case).

\begin{cor}  %%%%   Cor. 7.10
Let $\Bla, \Bmu \in \SP_{n,r}$. 
Assume that $\th_1 \gg \th_2 \gg \cdots \gg \th_m > 0$. 
Then $K^{\pm}_{\Bla+\Bth, \Bmu + \Bth}(\Bt)$ has a stable value, independent of the 
choice of $\Bth$, and we have
\begin{equation*}
K^{\pm}_{\Bla + \Bth, \Bmu + \Bth}(\Bt) = L_{\pm}(c(\Bla) - c(\Bmu);\Bt).
\end{equation*} 
\end{cor} 

\begin{proof}
The value $K^{\pm}_{\Bla + \Bth, \Bmu + \Bth}(t)$ can be 
expressed by the formula in Theorem 7.3 and Corollary 7.5.
(Note that in the ``$+$''case, (B) holds for $\Bmu + \Bth$ since $\th_m > 0$.) 
By our assumption 
$\th_1 \gg \th_2 \gg \cdots \gg \th_m$, the non-zero contribution only occurs in 
the case where $w = 1$.  Hence 
\begin{equation*}
K^{\pm}_{\Bla + \Bth, \Bmu + \Bth}(\Bt) = L_{\pm}(c(\Bla + \Bth) - c(\Bmu + \Bth);\Bt)
                                      = L_{\pm}(c(\Bla) - c(\Bmu);\Bt).
\end{equation*}  
The corollary is proved.
\end{proof}

\par

\par\vspace{1cm}
\noindent
T. Shoji \\
Department of Mathematics, Tongji University \\ 
1239 Siping Road, Shanghai 200092, P. R. China  \\
E-mail: \verb|shoji@tongji.edu.cn|

\end{document}